\documentclass[sn-mathphys,Numbered]{sn-jnl}

\usepackage{amsfonts,amsmath,amsthm}
\usepackage{geometry}
\usepackage{xcolor}
\usepackage{verbatim}
\usepackage{textcomp}
\usepackage{manyfoot}%
\usepackage{dsfont}
\usepackage{amssymb}
\usepackage{mathrsfs}
\usepackage{cases}

\pdfoutput=1
\usepackage{algorithm}%
\usepackage{algorithmicx}%
\usepackage{algpseudocode}%
\usepackage{listings}%

\usepackage{setspace,url}
\usepackage{graphics,graphicx, epstopdf}
\usepackage{float}
\usepackage{subfig}
\usepackage{tabularx}
\usepackage{longtable}
\usepackage{booktabs,multirow,array,multicol}
\usepackage{enumitem}

\theoremstyle{thmstyleone}%
\newtheorem{theorem}{Theorem}[section]
\newtheorem{lemma}[theorem]{Lemma}
\newtheorem{assumption}{Assumption}

\numberwithin{equation}{section}
\usepackage{setspace,url}
\usepackage{graphics,graphicx,epstopdf}
\usepackage{float}
\usepackage{subfig}
\usepackage{tabularx}
\usepackage{longtable}
\usepackage{booktabs,multirow,array,multicol}
\usepackage{enumitem}
\newtheorem{remark}{Remark}%

\newtheorem{definition}{Definition}%

\newcommand{\R}{\mathbb{R}}

\newcommand{\Prob}{\mathbb{P}}

\newcommand{\con}{\mathrm{con}}

\newcommand{\ewmin}{\lambda_{\min}}
\newcommand{\cA}{\mathcal{A}}
\newcommand{\cC}{\mathcal{C}}

\newcommand{\cF}{\mathcal{F}}

\newcommand{\cG}{\mathcal{G}}

\newcommand{\cN}{\mathcal{N}}
\newcommand{\cL}{\mathcal{L}}
\newcommand{\cS}{{\mathcal{S}}}
\newcommand{\cO}{{\mathcal{O}}}

\newcommand{\hX}{\mathbf{X}}
\newcommand{\hx}{\mathbf{x}}

\newcommand{\hz}{\mathbf{z}}
\newcommand{\hu}{\mathbf{u}}
\newcommand{\hp}{\mathbf{p}}
\newcommand{\hq}{\mathbf{q}}

\newcommand{\prox}{\mathrm{prox}}
\newcommand{\dist}{\mathrm{dist}}
\newcommand{\crit}{\mathrm{crit}}
\newcommand{\cond}{\mathrm{cond}}
\newcommand{\Exp}{\mathbb{E}}
\newcommand{\Sto}{\widetilde{\nabla}}

\newcommand{\iprod}[2]{\langle #1, #2 \rangle}
\DeclareMathOperator*{\argmin}{arg\,min}

\newcommand{\be}{\begin{equation}}
\newcommand{\ee}{\end{equation}}
\newcommand{\bee}{\begin{equation*}}
\newcommand{\eee}{\end{equation*}}

\begin{document}

\title[Article Title]{Dynamical convergence analysis for nonconvex linearized proximal ADMM algorithms}


\author[1,2]{\fnm{Jiahong } \sur{Guo}}\email{gjh0722@mail.dlut.edu.cn}

\author*[2]{\fnm{Xiao} \sur{Wang}}\email{wangx07@pcl.ac.cn}

\author[1]{\fnm{Xiantao} \sur{Xiao}}\email{xtxiao@dlut.edu.cn}

\affil[1]{\orgdiv{School of Mathematical Sciences}, \orgname{Dalian University of Technology}, \orgaddress{\city{Dalian}, \postcode{116024},  \country{China}}}

\affil*[2]{
\orgname{Peng Cheng Laboratory}, \orgaddress{\city{Shenzhen}, \postcode{518066}, \country{China}}}



\abstract{The convergence analysis of optimization algorithms using continuous-time dynamical systems has received much attention in recent years. In this paper, we investigate applications of these systems to analyze the convergence of linearized proximal ADMM algorithms for nonconvex composite optimization, whose objective function  is the sum of a continuously differentiable function and a composition of a possibly nonconvex function with a linear operator. We first derive a first-order differential inclusion for the linearized proximal ADMM algorithm, LP-ADMM. Both the global convergence and the convergence rates of the generated trajectory are established with the use of Kurdyka-\L{}ojasiewicz (KL) property. Then, a stochastic variant, LP-SADMM, is  delved into an investigation for finite-sum nonconvex composite problems. Under mild conditions, we obtain the stochastic differential equation corresponding to LP-SADMM, and demonstrate the almost sure global convergence of the generated trajectory by leveraging the KL property. Based on the almost sure convergence of trajectory, we construct a stochastic process that converges almost surely to an approximate critical point of objective function, and derive the expected convergence rates associated with this stochastic process.  Moreover, we propose an accelerated LP-SADMM that incorporates Nesterov's acceleration technique. The continuous-time dynamical system of this algorithm is modeled as a second-order stochastic differential equation. Within the context of KL property, we explore the related almost sure convergence and expected convergence rates.
}




\maketitle

\section{Introduction}\label{sec:intro}

The composite optimization problem 
\be\label{eq:p}
\min_{x\in \R^n} H(x):=f(x)+h(Ax), 
\ee
arises in various applications, including signal processing,  image processing, machine learning, and statistics. 
Here, $f:\R^n\to\R$ is a  continuously differentiable and possibly nonconvex function, $A:\R^n\to\R^m$ is a linear operator, and $h:\R^m\to\R$ is a simple and  possibly  nonconvex function that is commonly referred to as the \textit{regularizer}. The regularizer is used to  guarantee certain desirable properties of the solution.  
Popular regularizers in the literature include the convex $\ell_1$, $\ell_2$ and total variation, as well as weakly convex  regularizers such as the  minimax concave penalty (MCP) and smoothly clipped
absolute deviation (SCAD), and the nonconvex   $\ell_p$ ($0\le p<1$) regularizer. 

Problem \eqref{eq:p} in the convex setting has been extensively studied, and various  primal-dual algorithms have been proposed, such as primal-dual hybrid gradient method,  alternating direction method of multipliers (ADMM) and their  variants. Related references  include \cite{2010Esser, 2011Boyd, Chen2014,  Teboulle2014,  RDN2020,2021Liuxuyin}. However, it is challenging  to directly extend these algorithms   to    problem \eqref{eq:p} in the  fully nonconvex setting, due to  the  nonconvexity of both $f$ and $h$. To address this challenge, recent study has focused on the Kurdyka-\L{}ojasiewicz (KL) property, which provides a tool for analyzing the global  convergence of optimization algorithms in the nonconvex setting. Relevant work includes but not limited to \cite{AB2009, ABRS2010, ABS2013,BST2014,TSP2018,Csaba2021}. 
Li and Pong \cite{Li2015} applied ADMM to solve the nonconvex problem \eqref{eq:p}. They showed that under the assumptions that $f$ and $h$ are semialgebraic, and $A$ is surjective,  the iterates generated by ADMM converge to a critical point of the objective function $H$.  
In \cite{GM2018},  a class of majorization-minimization methods for solving \eqref{eq:p} with  $A$ being a nonlinear operator were  investigated. The global convergence  of the iteration sequences generated by these methods were established under the KL inequality.  Bolte et al. \cite{BST2018} proved that the bounded iteration sequence generated by some  Lagrangian-based methods, including proximal multipliers method and proximal ADMM, is globally convergent to a critical point in the semialgebraic setting. Under the assumption  that the associated augmented Lagrangian has the KL property, Bo{\c{t}} and Nguyen  \cite{BN2020} proved that the iterates of proximal ADMM converge to a Karush-Kuhn-Tucker point, and established  convergence rates for the sequences of  the augmented Lagrangian values  and iterates with \L{}ojasiewicz exponent.

Recently, there has been a lot of  attention paid to 
the behavior of optimization algorithms  from the viewpoint of  dynamical systems. For instance, the classic gradient descent method for minimizing a differential function can be regarded as the Euler discretization of a first-order dynamics.  Su et al.   \cite{SBC2016} showed that the exact limit of Nesterov's  accelerated gradient method \cite{Nesterov1983AMF} is a second-order differential equation which provides a different perspective to understand and analyze the accelerated methods.  Attouch et al.  \cite{Attouch2018FastCO} considered a perturbed second-order dynamical
system with asymptotic vanishing viscosity. It can be interpreted as a continuous version of some fast convergent methods (e.g., FISTA  \cite{2009Beck}) for minimizing a convex smooth function. The authors derived that the objective function with respect to the generated trajectory exhibits a sublinear convergence.  
In  \cite{Franca2018}, it was shown that the continuous limits of ADMM and its accelerated variant  for solving \eqref{eq:p} in the convex and differentiable setting are a first-order and second-order differential equations, respectively. And  the convergence rates of these dynamical systems were obtained. Further, an extension from results in \cite{Franca2018}   to the nonsmooth constrained problems in the convex and strongly convex settings  was proposed in \cite{Guilherme2023}.  He et al. \cite{HHF2021} presented a  second-order inertial primal-dual  dynamical system for a separable convex optimization problem with linear equality constraints and analyzed the asymptotic properties of this system. A more general second-order inertial system with damping, which can be viewed as the inertial continuous counterpart of ADMM-type methods, was studied in \cite{ACFR2022}. 
   
Motivated by existing work on dynamic analysis of numerical optimization algorithms, we explore the  behavior of a linearized proximal ADMM (LP-ADMM)  for the fully nonconvex composite problem \eqref{eq:p} by studying  the  trajectory generated  by the corresponding dynamical  system. Under the  KL inequality, we demonstrate the global convergence of the generated trajectory. We further study a stochastic variant of LP-ADMM called LP-SADMM for solving the following finite-sum problem:
  \be\label{eq:finite-p}
\min_{x\in\R^n}H(x):=f(x)+h(Ax) \quad\mbox{with } f(x):=\frac{1}{N}\sum_{i=1}^N f_i(x),
\ee
where $h$ is weakly convex and nonsmooth, and $N$ can be large. 
Moreover, an accelerated LP-SADMM, which combines the Nesterov's  accelerated gradient with LP-SADMM, is proposed and studied for \eqref{eq:finite-p}. 
 
\textbf{Contributions.}
The main contributions of this paper are summarized as follows.
\begin{itemize}
  \item[$\bullet$] We study  a linearized proximal ADMM (LP-ADMM) method for the fully nonconvex and nonsmooth composite optimization problem \eqref{eq:p}. By using  Taylor's theorem and under certain  assumptions, we derive a first-order continuous dynamical system, which can be seen as a continuous limit of LP-ADMM:
  \[
 0\in\lambda\dot{\hx}(t)+\partial H(\hx(t)),\quad \text{ with } \hx(0)=x^0,
\]
where $\lambda>\|A^TA\|_2$.
We then analyze the behavior of LP-ADMM by taking advantage of this dynamical system. Under an assumption that the objective function $H$ admits  the chain rule,  we show a key descent property of the composition of $H$ and a generated trajectory $\hx$ with respect to time. Based on this descent property and  a condition that the trajectory is bounded, we establish a subsequence convergence of the trajectory generated by the continuous dynamical system. Furthermore, when the objective function is a KL function, we prove that the trajectory converges to the critical point of $H$.  Moreover, in the context of KL property with \L{}ojasiewicz exponent,  we provide the convergence rates of trajectory and objective function values.
  	
 \item[$\bullet$] 
We introduce a stochastic variant of LP-ADMM, called stochastic linearized proximal ADMM (LP-SADMM), to address the challenges associated with computing the full gradient of $f$ when solving the finite-sum problem \eqref{eq:finite-p}. Leveraging the smoothness properties of the Moreau envelope for weakly convex functions and under certain assumptions, we derive a first-order stochastic differential equation that serves as a continuous weak approximation of LP-SADMM:
\[0=\lambda\dot{\hx}(t)+\nabla H_{\mu}(\hx(t))+\rho^{-1/2}\dot{W}(t), \quad \text{with } \hx(0)=x^0,\]
where $\rho>0$ denotes a penalty parameter, $W$ represents a Brownian motion (or Wiener process), and $H_\mu(x):= f(x)+h_\mu(Ax)$ with $\mu>0$ and $h_\mu$ being the $\mu$-Moreau envelope of $h$. By employing a descent property with respect to the expectation of $H_{\mu}$, we establish the convergence of the generated trajectory $\hx$ in the almost sure sense. Assuming that $H_\mu$ is a KL function with a \L{}ojasiewicz exponent, we demonstrate the almost sure convergence of the trajectory $\hx$  to a critical point of $H_\mu$. Utilizing the relationship between the critical points of $H$ and $H_{\mu}$, we prove that $\bar{\hx}$, defined by
\[
\bar{\hx}(t):=\hx(t)-A^T(AA^T)^{-1}(A\hx(t)-\prox_{\mu h}(A\hx(t))), \quad t\ge0,
\]
converges to an approximate critical point of $H$ for sufficiently small constant $\mu$. Additionally, we provide the expected convergence rates of $\bar{\hx}$.
 
\item[$\bullet$] Motivated by the Nesterov's  accelerated gradient method, we propose an  accelerated variant of LP-SADMM for problem \eqref{eq:finite-p}. 
We first derive the continuous weak approximation of this proposed method. The approximation takes the form of a second-order stochastic   differential equation:
\[
0=\lambda\ddot{\hx}(t)+\lambda(\gamma+\frac{\alpha}{t})\dot{\hx}(t)+\nabla H_{\mu}(\hx(t)) +\rho^{-1/4}\dot{W}(t), \quad \text{with } \hx(0)=x^0,
\]
where $\alpha,\gamma>0$. To analyze the convergence of the trajectory of the second-order stochastic  differential equation, we establish a descent property with respect to the expectation of a Lyapunov function under the condition that the generated trajectory $[\hx;\dot{\hx}]$ is bounded. 
 We further prove that $\dot{\hx}$ converges to zero almost surely and any convergent subsequence of $\hx$ converges almost surely to a critical point of the corresponding Lyapunov function. Moreover, if the Lyapunov function is a KL function with \L{}ojasiewicz exponent, we demonstrate the global almost sure convergence of the generated trajectory. Using the relationship between the critical points of $H$ and $H_\mu$, we establish the expected convergence rates of $\bar{\hx}$ to an approximate critical point of $H$ for sufficiently small $\mu$.
\end{itemize}

\textbf{Organization.}
The rest of this paper is organized as follows.   Some notations and preliminaries are introduced in Section \ref{sec:notation}. In Section \ref{sec:Linearized proximal ADMM}, we present a linearized proximal ADMM algorithm for problem \eqref{eq:p}, and derive its continuous-time first-order  dynamical system with convergence analysis. In Section \ref{sec:sto-ADMM}, we propose  a stochastic variant algorithm, LP-SADMM, for problem \eqref{eq:finite-p}. By analyzing its corresponding first-order stochastic differential equation, we establish  convergence properties of the proposed algorithm. In  Section  \ref{sec:acce-ADMM}, we propose an accelerated variant of LP-SADMM and present its convergence analysis from the perspective of  a second-order stochastic differential equation. Finally, we give some conclusions and discussions.

\section{Notations and preliminaries} \label{sec:notation}
Take $\R^n$ and $\R^m$ as Euclidean spaces equipped with  the standard inner product $\iprod{\cdot}{\cdot}$ and  associated norm $\|\cdot\|=\sqrt{\iprod{\cdot}{\cdot}}$.
For a linear operator $A:\R^n\rightarrow\R^m$, its norm is denoted by
\[
\|A\|:=\max\{\|Ax\|:x\in\R^n\ \mbox{with}\ \|x\|\leq 1\}.
\]
Let $F:\R^n \rightrightarrows\R^m$ be a set-valued mapping. 
$F$ is said to be \emph{outer semicontinuous} at $\bar{x}$ if, for any $v\in\R^m$ such that there exist $x^k\rightarrow\bar{x}$ with $v^k\in F(x^k)$ and $v^k\rightarrow v$, it holds that $v\in F(\bar{x})$. Given a closed set $\cC \subseteq \mathbb R^n$, we denote a \emph{convex hull} of $\cC$  by $\con\, \cC$, and  the \emph{distance} between a point $x\in\mathbb R^n$ and $\cC$ by  $\dist(x,\cC):=\min_{y}\{\|x-y\|: y\in\cC\}$. 

\subsection{Clarke subgradient}\label{subsec:subdifferential}
Let $f:\R^n\to\R$ be locally Lipschitz continuous on an open set $\cO\subset\R^n$, and let $\cC$ be a subset of $\cO$ such that $f$ is differentiable at each point of $\cC$. Then by \cite[Theorem 8.49]{RW1998} and \cite[Theorem 9.61]{RW1998}, the Clarke subgradient set  of $f$ at $\bar{x}\in\cO$ can be expressed as 
\[
\partial f(\bar{x}):=\con\{v:\exists\, x\to\bar{x} \text{ with } x\in\cC, \nabla f(x)\to v\}.
\] 
Moreover,  $\partial f(\bar{x})$ is nonempty, convex and compact for any $\bar{x}\in\cO$. Additionally, $\partial f$ is outer semicontinuous and locally bounded on  $\cO$, and see  \cite{RW1998} for these conclusions. We denote the set of \emph{critical points} of  $f$ by 
\[
\crit f:= \{x\in\R^n: \dist(0,\partial f(x))=0\}.
\] 
Moreover, if $f$ is a differentiable function, then $\crit f$ becomes the set  of  all points satisfying the condition $\nabla f(x)=0$. 
Given two locally Lipschitz continuous functions $f_1$, $f_2:\mathbb R^n \to \mathbb R$, if one of them is differentiable at a point $x\in\R^n$, then it holds that
\be\label{eq:calculus}
\partial(f_1+f_2)(x)=\partial  f_1(x)+\partial  f_2(x).
\ee 
By \cite[Theorem 2.3.10]{Clarke1985}, when a locally Lipschitz continuous function $f:\R^m\to\R$ is convex or the linear operator $A:\R^n\to\R^m$ is surjective, the composition of $f$ and $A$ admits the following chain rule
\be\label{eq:chain-composition}
\partial (f\circ A)(x)=A^T\partial f(Ax), \   \forall x\in\R^n.
\ee

\subsection{Chain rule}
Let	 $\hx:[0,+\infty)\to\R^n$ be an absolutely continuous function and $F:\R^n \rightrightarrows\R^n$ be a set-valued mapping.  If  the following first-order differential inclusion holds:
\be\label{eq:first-inclusion}
0\in \dot{\hx}(t)+F(\hx(t)),
\ee
for almost every (for short, a.e.) $ t \ge 0$,
we call $\hx$  a \emph{trajectory}  of \eqref{eq:first-inclusion}. 
\begin{definition}\label{def:chain-rule1} 
Let	 $\hx:[0,+\infty)\to\R^n$ be an absolutely continuous function and $f:\R^n\to\R$ be locally Lipschitz continuous. We say that a chain rule holds for $f$ if for almost every $t\geq 0$, 
	\be\label{eq:chain-rule1}
    \frac{d}{dt} f(\hx(t))=\langle \partial f(\hx(t)), \dot{\hx}(t)\rangle=\{\langle v, \dot{\hx}(t)\rangle: \forall v\in \partial f(\hx(t))\}.
	\ee
\end{definition}
The chain rule \eqref{eq:chain-rule1} can be satisfied by the two  classes of functions, as  presented in \cite[Section 5]{DDKL2020}: (i) any locally Lipschitz continuous function that is subdifferential regular; (ii) any locally Lipschitz continuous function that is Whitney stratifiable. The latter one includes a wide variety of functions, such as  semialgebraic functions and functions definable in an $o$-minimal structure.

\subsection{Brownian motion  and  stochastic differential equation} \label{sec:brownian}

We denote by $\cN(0,s)$ the Gaussian distribution with a mean value of zero and variance of $s$.
Let $\{W(t), t\geq 0\}$ be a standard  \emph{Brownian motion} (or Wiener process) defined on a probability space $(\Omega,\cF,\Prob)$, where $\Omega$ is a sample space equipped with a $\sigma$-algebra  $\cF$ and a probability measure $\Prob$. The Brownian motion is a stochastic process and satisfies $W(t)-W(s)\sim \cN(0,t-s)$ for any $ t>s\geq 0$. Although 
$W(t)$ is nondifferentiable at any point, the formal derivative $\dot{W}(t)=dW/dt$ is elaborated in \cite{Bezandry2011, Evans2012, Ikeda1981} and is widely used in stochastic differential equation.  

Let $G:\R\to\R^n$ be a real-valued measurable stochastic process. It is said that $G\in \mathbb L^p(a, b)$ for $p\geq 1$ and $b>a\ge 0$ if $G$ satisfies 
$\int_{a}^{b}{\|G(t)\|^p}dt<+\infty$,  and $G\in\mathbb L_e^p(a, b)$ if  
$\Exp\left[\int_{a}^{b}{\|G(t)\|^p}dt\right]<+\infty$. 
For a real-valued measurable stochastic process $G\in\mathbb L_e^2(0, T)$, we call  
\[\int_{0}^{T}{G(t)}dW(t)=\lim_{n\to\infty}\int_{0}^{T} G_{n}(t)dW(t)\]
the \emph{It{\^o}'s integral} \cite[Definition 3.29]{Bezandry2011} of $G$ with respect to $ W(t)$ on the interval $[0,T]$ for $T> 0$, where $\{G_n\}$ is a sequence of simple stochastic process such that 
\[\lim_{n\to\infty}\int_{0}^{T}|G(t)-G_{n}(t)|^2dt=0.\]  
The It{\^o}'s integral has the following property
\[
\Exp\left[\int_{0}^{T}{G(t)}dW(t)\right]=0.
\]

We call the stochastic process $\hx:[0,+\infty)\to\R^n$ a \emph{trajectory} of  the first-order stochastic differential equation 
\be\label{def:first-SDE}
0= \dot{\hx}(t)+ F(\hx(t))+\beta\dot{W}(t), \text{ for } 0\leq t\leq T, 
\ee
 if  $\hx$ is progressively measurable and almost surely (a.s.) continuous, $F(\hx)\in \mathbb L_e^1(0, T)$, and the following equation  holds
\[\hx(t)=\hx(0)+\int_{0}^tF(\hx(s))ds+\int_{0}^t\beta\dot{W}(s)ds \text{ a.s. for all } 0\leq t\leq T,\]
where $\beta>0$ is a constant and $T$ is a any time. 
We would like to note that the trajectory $\hx$, determined by \eqref{def:first-SDE}, depends on a specific sample $\omega\in \Omega$.  Furthermore, the continuity of $\hx$ with probability 1 implies the existence of an event $\cA\in\cF$ with $\Prob(\cA)=1$, such that for any fixed $\omega\in\cA$, $\hx(t,\omega)$ is continuous with respect to $t$.  
 In this stochastic setting, let $\phi(x)$ be a twice continuously differentiable function. Then the It{\^o}'s  formula \cite{Evans2012} holds: 
 \be\label{eq:ito-lemma}
 \begin{aligned}
\frac{d\phi(\hx(t))}{dt}=\left(-\nabla\phi^T F(\hx(t))
 +\frac{\beta^2}{2}\nabla\cdot\nabla\phi\right)-\beta \nabla\phi^T\dot{W}(t)=\langle \nabla\phi,\dot{\hx}(t)\rangle+\frac{\beta^2}{2} \nabla\cdot\nabla\phi.
 \end{aligned}
\ee
where $\nabla\cdot\nabla$ is Laplacian. 


For a second-order stochastic differential equation
\be\label{def:second-SDE}
0= \ddot{\hx}(t)+\gamma(t)\dot{\hx}(t)+ F(\hx(t))+\beta\dot{W}(t),
\ee
we set $\hX(t):=[\hX_1(t);\hX_2(t)]$ where $\hX_1(t)=\hx(t)$ and $\hX_2(t)=\dot{\hx}(t)$, then \eqref{def:second-SDE} can be arranged into a first-order stochastic differential equation
\[
0= \dot{\hX}(t)+[-\hX_2(t);\gamma(t)\hX_2(t)+F(\hX_1(t))]+\beta[0;I]\dot{W}(t).
\]
Hence, $\hX(t)$ is  defined  as a  trajectory of the second-order stochastic differential equation \eqref{def:second-SDE}. 


\subsection{Moreau envelop and weak convexity}\label{sec:moreau and weak convex}
Let $f: \R^n\to \R$ be  proper and lower semicontinuous. Given $\mu>0$, we define the \emph{proximal mapping} and $\mu$-\emph{Moreau envelop} of $f:\R^n\to \R$ as
\[\prox_{\mu f}(y):=\argmin_{x\in\R^n}\left\{f(x)+\frac{1}{2\mu}\|x-y\|^2\right\}\]
and 
\[f_{\mu}(y):=\min_{x\in\R^n}\left\{f(x)+\frac{1}{2\mu}\|x-y\|^2\right\},\]
respectively. 
It is indicated in \cite[Theorem 1.25]{RW1998}  that 
for any $y\in\R^n$,
\be\label{eq:property-moreau1}
f_{\mu}(y)\nearrow f(y) \text{ as } \mu\searrow 0.
\ee

A
function $f:\R^n\to \R$ is said to be \emph{$\varrho$-weakly convex} if the mapping $x\longmapsto f(x)+\frac{\varrho}{2}\|x\|^2$ is convex. 
By applying the definition of  Moreau envelope and \eqref{eq:property-moreau1}, we  obtain $-\infty < \inf f(y) \leq f_{\mu}(y) \leq f(y)$ for any $y\in \R^n$. Moreover, as shown in \cite[Corollary 3.4]{Hoheisel2020}, $f_{\mu}$ is convex and $\max\{\frac{1}{\mu}, \frac{\varrho}{1-\varrho\mu}\}$-smooth, provided that $\mu < 1/\varrho$.

%

\section{LP-ADMM algorithm and convergence analysis}\label{sec:Linearized proximal ADMM}
In this section, we delve into the analysis of a linearized proximal ADMM algorithm (LP-ADMM) for solving problem \eqref{eq:p}. We first demonstrate that the continuous limit of this algorithm is a first-order differential inclusion. We then establish the global  convergence of objective function values and trajectory, along with the  convergence rates under some mild conditions. 

\subsection{LP-ADMM and continuous-time system }\label{sec:continuous limit-det}
Consider an equivalent form of \eqref{eq:p}:
\[
\min_{x\in\mathbb R^n, z\in\mathbb R^m} f(x) + h(z)\, \mbox{ s.t. } z=Ax.
\]
The augmented Lagrangian associated with this problem is defined in terms of a penalty parameter $\rho>0$ and a vector of multipliers $u\in\R^m$,  given by
\bee
L_{\rho}(x,z,u):=f(x)+h(z)+\langle u, Ax-z\rangle+\frac{\rho}{2}\|Ax-z\|^2.
\eee
Augmented Lagrangian-based algorithms  have been extensively studied  in the literature, with notable works including \cite{2011Boyd, Teboulle2014, BST2018, RDN2020}. An iteration of  classical ADMM,  as described in \cite{2011Boyd}, can be expressed as  
\begin{subnumcases}{\label{eq:classical-iter}}
	x^{k+1}\in\argmin_{x\in\R^n} \left\lbrace L_{\rho}(x,z^k,u^k)=f(x)+\langle u^k, Ax-z^k\rangle+\frac{\rho}{2}\|Ax-z^k\|^2 \right\rbrace, \label{eq:classical-iter-1}\\
	\notag\\
	z^{k+1}\in\argmin_{z\in\R^m}\left\lbrace L_{\rho}(x^{k+1},z,u^k)=h(z)+\langle u^k, Ax^{k+1}-z\rangle+\frac{\rho}{2}\|Ax^{k+1}-z\|^2\right\rbrace, \label{eq:classical-iter-2}\\
	\notag\\
	u^{k+1}=u^k+Ax^{k+1}-z^{k+1}.\label{eq:classical-iter-3}
\end{subnumcases} 
It is worth noting that an inherent challenge arises in updating $x^{k+1}$ due to the presence of the coupled term $Ax$. Meanwhile, the potential nonlinearity and nonconvexity of $f$ can make it challenging to compute the proximal mapping of $f$. To address these issues, a linearization approach with respect to $x^k$ in \eqref{eq:classical-iter-1} is used to obtain a more tractable solution:
\be\label{eq:modified-iter}
x^{k+1}\in\argmin_{x\in\R^n} \left\{  \langle\nabla f(x^k),x-x^k\rangle+\langle u^k,Ax-z^k\rangle+\frac{\rho}{2}\|Ax-z^k\|^2+\frac{1}{2\eta}\|x-x^k\|_{M}^2 \right\}.
\ee
Here, $M$ is a symmetric positive definite matrix  and $\|x\|_M^2:=x^TMx$. Then \eqref{eq:modified-iter} admits a unique solution. From the optimality condition for \eqref{eq:modified-iter} it follows that 
\[
x^{k+1}=-(\rho A^TA+\frac{1}{\eta}M)^{-1}\left( \nabla f(x^k)+A^T u^k-\rho A^Tz^k-\frac{1}{\eta}Mx^k\right).
\] 
To ease the computational burden  caused by a possibly dense matrix $\rho A^TA+\frac{1}{\eta}M$, we  set $M=\eta(\tau I-\rho A^TA)$,  obtaining Algorithm \ref{alg:Modified ADMM}. 

\begin{algorithm}
\caption{LP-ADMM}\label{alg:Modified ADMM}
\begin{algorithmic}[1]
\Require Initial point $(x^0,z^0,u^0) \in \R^n\times\R^m\times\R^m$,  parameters $\eta,\rho>0$ and $\tau> \rho\|A^TA\|+1/\eta$.
\For{ $k=0,1,2,\ldots$ } 
\State Update $x^{k},z^{k},u^{k}$ as follows:
\begin{subnumcases}{\label{eq:iter}}
				x^{k+1}=x^k-\frac{1}{\tau}\left(\nabla f(x^k)+\rho A^T(Ax^k-z^k+\frac{1}{\rho}u^k)\right),\label{eq:iter1}\\
				\notag\\
				z^{k+1}\in\argmin_{z\in\R^m}\left\lbrace h(z)+\langle u^k, Ax^{k+1}-z\rangle+\frac{\rho}{2}\|Ax^{k+1}-z\|^2\right\rbrace,\label{eq:iter2}\\
				\notag\\
				u^{k+1}=u^k+Ax^{k+1}-z^{k+1}\label{eq:iter3}.
			\end{subnumcases}
\State  Set $k \Leftarrow k+1$.
\EndFor 
\end{algorithmic}
\end{algorithm}
The setting of parameter $\tau$ in Algorithm \ref{alg:Modified ADMM},  $\tau>\rho|A^TA|+1/\eta$, ensures the positive definiteness of $M$. In addition, $1/\tau$ can be interpreted as the step size for the gradient step. Algorithm \ref{alg:Modified ADMM} has been previously studied in \cite{RDN2020, lan2015}. However, in contrast to these studies, our focus in this section lies in the continuous dynamic analysis of the algorithm specifically for nonconvex objectives.

Before proceeding,  we first lay out some standing and standard assumptions that will be  used throughout the remainder of this paper.

\begin{assumption} \label{ass:det} 
For problem \eqref{eq:p}, suppose that
\begin{itemize} 
\item[(i)] Lower boundedness holds, i.e. $\inf_{x\in\R^n} f(x)>-\infty$,   $\inf_{y\in\R^m} h(y)>-\infty$.
\item[(ii)] Function $f$ is continuously differentiable,  and $h$ is locally Lipschitz continuous.
\item[(iii)] Linear operator $A$ is surjective.
\end{itemize}
\end{assumption}
\begin{remark} \label{re:surjective} Here are some comments on Assumption \ref{ass:det}.
\begin{itemize} 
\item[(i)] Assumption \ref{ass:det}(i) ensures that problem  \eqref{eq:p}  is well-defined, and plays a significant role in the weak convergence analysis, as demonstrated in Theorems \ref{th:weak-convergence}, \ref{th:sto-weak-convergence} and \ref{th:weak-conver1-acce}.

\item[(ii)] Note that  $A$ is surjective if and only if $AA^T$ is positive definite. Additionally,  
Assumption \ref{ass:det}(iii) provides a sufficient condition for the chain rule to hold:
\be\label{eq:chain-h}
\partial (h\circ A)(x)=A^T\partial h(Ax),
\ee
which is crucial in  deriving the continuous differential inclusion  of Algorithm \ref{alg:Modified ADMM}. 

\end{itemize}	
\end{remark}
We now proceed with an informal derivation to establish a first-order differential inclusion  for Algorithm \ref{alg:Modified ADMM}. This system can be regarded as the continuous limit of Algorithm \ref{alg:Modified ADMM} when the parameter $\rho$ approaches infinity. The derived system forms the foundation of our analysis in this section. Let $\{x^k\}, \{z^k\}, \{u^k\}$ be generated by Algorithm \ref{alg:Modified ADMM}. The optimality condition for  \eqref{eq:iter2} implies
\be\label{eq:optimal-2}
0\in \partial h(z^{k+1})-\rho(Ax^{k+1}-z^{k+1}+\frac{1}{\rho}u^k),
\ee
which combines \eqref{eq:iter1} to yield  
\be\label{eq:optimal-2.1}
0\in\tau(x^{k+1}-x^k)+\nabla f(x^k)+A^{T}\partial h(z^{k+1})-\rho A^TA(x^{k+1}-x^{k})+\rho A^T(z^{k+1}-z^{k}).
\ee
Following a similar approach as presented in  \cite{SBC2016},  we introduce the \emph{Ansatz}: $x^k\approx\hx(k/\rho), z^k\approx\hz(k/\rho), u^k\approx\hu(k/\rho)$ for some continuously differentiable functions $\hx,\hz,\hu$ defined on $[0,+\infty)$. Denote $s=1/\rho$ and $k=t/s$. 
Then $\hx(t)\approx x^{t/s}=x^k, \hx(t+s)\approx x^{(t+s)/s}=x^{k+1}$ as the parameter $s$ is sufficiently small.  This approximation holds true for $\hz, \hu$.
By applying Taylor's theorem we derive the following relations
\begin{align}
& u^{k+1} =\hu(t)+s\dot{\hu}(t)+O(s^2),\label{eq:differential mean-u}\\
& z^{k+1} =\hz(t)+s\dot{\hz}(t)+O(s^2),\label{eq:differential mean-z}\\
& x^{k+1}=\hx(t)+s\dot{\hx}(t)+O(s^2),\label{eq:differential mean-x}
\end{align}
which indicate from    \eqref{eq:iter3} that
\bee
s \dot{\hu}(t)=A\hx(t)-\hz(t)+s(A\dot{\hx}(t)-\dot{\hz}(t))+O(s^2).
\eee
Note that as $s\to 0$, we have 
\be\label{eq:condition1}
\begin{aligned}
A\hx(t)=\hz(t),
\end{aligned}
\ee
for any $t\in[0,+\infty)$, which further indicates from the arbitrariness of $t$ that 
\be\label{eq:condition2}
A\dot{\hx}(t)=\dot{\hz}(t). 
\ee
By inserting \eqref{eq:differential mean-z} and \eqref{eq:differential mean-x} into  \eqref{eq:optimal-2.1},  it yields  
\be\label{eq:inclusion2}
0\in\tau(s\dot{\hx}(t)+O(s^2))+\nabla f(\hx(t))+A^{T}\partial h(\hz(t)+O(s))-A^TA\dot{\hx}(t)+A^T\dot{\hz}(t).
\ee
Recall that $\tau s> \|A^TA\|+s/\eta$ from the setting of $\tau$ in Algorithm \ref{alg:Modified ADMM} and $s=1/\rho$. Notably, there exists a constant $\lambda>\|A^TA\|$ such that $\tau s=\lambda+O(s)$. Then 
\eqref{eq:inclusion2} is rearranged into 
\be\label{eq:inclusion3}
0\in\lambda\dot{\hx}(t)+O(s)+\nabla f(\hx(t))+A^{T}\partial h(\hz(t)+O(s))-A^TA\dot{\hx}(t)+A^T\dot{\hz}(t).
\ee
Let us consider the limit of above inclusion as $s\to 0$. Since $\hz(t)+O(s)\to A\hx(t)$ as $s\to 0$ from \eqref{eq:condition1},  it is worthy to note that 
for any vector $d(s)\in A^{T}\partial h(\hz(t)+O(s))$ with  $d(s)\to d$  as $s\to 0$, 
it must hold that $d\in A^{T}\partial h(A\hx(t))$  by the outer semicontinuity of $\partial h$. Hence, taking limit of \eqref{eq:inclusion3} as $s\to 0$, and using   \eqref{eq:condition1} and \eqref{eq:condition2} yield 
\be\label{eq:inclusion}
 0\in\lambda\dot{\hx}(t)+\partial H(\hx(t)) \text{ with } \hx(0)=x^0,
\ee
where $\lambda>\|A^TA\|$.

\begin{remark}
By \cite[Theorem 2.1]{Aubin1984chapter2}, if $\partial H$ is nonempty, closed convex valued and outer semicontinuous, and there exists a  constant $C>0$ such that $\dist(0,\partial H(x))\leq C$ for any $x\in\R^n$,   then  
the differential inclusion \eqref{eq:inclusion} admits at least one trajectory.  According to  the properties of Clarke subgradient in Subsection \ref{subsec:subdifferential}, $\partial H$  clearly meets the above conditions.  Additionally, if there exists a constant $L>0$  such that for any $x,y\in\R^n$ and $u\in\partial H(x), v\in\partial H(y)$, 
\[\langle u-v,x-y\rangle\leq L\|x-y\|^2, \] 
then the trajectory of  \eqref{eq:inclusion} 
 is unique, as stated in  \cite[Theorem 2.2.2]{Kunze2000}.
\end{remark}

\subsection{Weak convergence analysis }
In Subsection \ref{sec:continuous limit-det}, we have obtained a differential inclusion  of Algorithm \ref{alg:Modified ADMM}. Assuming that the chain rule is always satisfied for the objective $H(\hx)$, we will   establish a  weak convergence property of the trajectory generated by  this continuous differential inclusion, under the condition that the trajectory is bounded, which is standard in the global convergence analysis of nonconvex optimization algorithms (see \cite{BST2014,BST2018,Radu2018, Radu2020} for instance).

\begin{assumption}\label{ass:chain-rule}
The objective function $H(\hx)=f(\hx)+h(A\hx)$ admits the  chain rule, as shown in Definition \ref{def:chain-rule1}.
\end{assumption}
Assumption \ref{ass:chain-rule}  has been employed in works for  nonconvex nonsmooth optimization, such as \cite{DDKL2020} and \cite{Camille2021}.   
When $H$ is a locally Lipschitz continuous subdifferential regular function or a locally Lipschitz Whitney stratifiable function, Assumption \ref{ass:chain-rule} is satisfied.

Before presenting the convergence properties of Algorithm \ref{alg:Modified ADMM}, we first provide a  descent property in terms of the objective function value that is essential in the subsequent analysis.

\begin{lemma}\label{lem:descent}
Suppose Assumptions \ref{ass:det} and \ref{ass:chain-rule} hold. Let $\hx$ be a trajectory of the differential inclusion \eqref{eq:inclusion}. 
Then for any $t_2\ge t_1\geq 0$, we have 
\be\label{eq:descent}
H(\hx(t_2))+\lambda^{-1}\int_{t_1}^{t_2}{\dist^2(0,\partial H(\hx(t)))}dt\leq H(\hx(t_1)).
\ee
\end{lemma}
\begin{proof}
Applying  the chain rule for $H(\hx(t))$ and $\lambda\dot{\hx}(t)\in-\partial H(\hx(t))$ from \eqref{eq:inclusion}, we have 
\be\label{eq:chain-rule-application}
\frac{d}{dt}H(\hx(t))=\langle \partial H(\hx(t)), \dot{\hx}(t)\rangle 
= -\lambda\|\dot{\hx}(t)\|^2, \text{ for } t\geq 0.
\ee
The relation $\lambda\dot{\hx}(t)\in-\partial H(\hx(t))$ also implies that 
\be\label{eq:dist-dot-x}
\dist(0,\partial H(\hx(t)))\leq \lambda\|\dot{\hx}(t)\|,
\ee
which, together with  \eqref{eq:chain-rule-application}, yields that for any $t\geq 0$, 
\be\label{eq:differential-dist}
\frac{d}{dt}H(\hx(t))\leq -\lambda^{-1}\dist^2(0,\partial H(\hx(t))).
\ee
Integrating both sides of the above inequality over $[t_1,t_2]$, we obtain 
\[
H(\hx(t_2))-H(\hx(t_1))\leq-\lambda^{-1}\int_{t_1}^{t_2}\dist^2(0,\partial H(\hx(t)))dt.
\]
This proof is completed.   	
\end{proof}

\begin{remark} Lemma \ref{lem:descent} indicates that the objective function $H(\hx(t))$ is nonincreasing. Additionally,   if $H(\hx(t_1))=H(\hx(t_2))$,  $\dist(0,\partial H(\hx(t)))=0$ for a.e.  $t\in[t_1,t_2]$. Therefore,   $\hx(t) \in\crit H$ for  a.e. $t\in[t_1,t_2]$.
\end{remark}

Consider   a trajectory $\hx$ of the  differential inclusion \eqref{eq:inclusion}. We say that  $x_{\infty}$ is a  cluster point of  $\hx(t)$ if  there exists an increasing sequence $t_k\to \infty$ such that $\hx(t_k)\to x_{\infty}$. 
Using the descent property derived in Lemma \ref{lem:descent}, we establish a preliminary convergence result. 

\begin{theorem}\label{th:weak-convergence}
	Let $\hx$ be a  trajectory of \eqref{eq:inclusion}, and $\cC$ be the set of cluster points of $\hx$. 
	Under Assumptions \ref{ass:det} and \ref{ass:chain-rule}, the following statements hold true:
\begin{itemize}
\item[(i)] $\int_{0}^{\infty}{\|\dot{\hx}(t)\|^2} dt<+\infty$;
\item[(ii)] $H(\hx)$  is finite and constant over $\cC$;
\item[(iii)] $\cC\subseteq\crit H$.
\end{itemize}
Moreover, if $\hx$ is bounded, then
\begin{itemize}
\item[(iv)] $\cC$ is nonempty, compact and 
\[\lim_{t\to +\infty}\dist((\hx(t)),\cC)= 0.\]
\end{itemize}
\end{theorem}
\begin{proof}
	Recall that, as shown in Lemma \ref{lem:descent}, $H(\hx(t))$ is nonincreasing. This, together with the lower boundedness  of $H(\hx(t))$ by Assumption \ref{ass:det}, yields that there exists a finite constant $\bar H$ such that
	\be\label{eq:objection converge}
	H(\hx(t))\to \bar{H}.
	\ee  Integrating \eqref{eq:chain-rule-application} over $[0,+\infty)$ and using \eqref{eq:objection converge}, we obtain  
	\[\int_{0}^{\infty}{\|\dot{\hx}(t)\|^2} dt<+\infty.\]
	Thus, item (i) holds.	
 
	For any $x_{\infty}\in \cC$, there exists an increasing sequence $t_k$ such that $\hx(t_k)\to x_{\infty}$. 
	Since $f$ is continuously differentiable and $h$ is locally Lipschitz continuous,  $H$ is continuous over $\R^n$. Therefore, it holds that
	\be\label{eq:obj-sub-converge}
	H(\hx(t_k))\to H(\hx_{\infty}).
	\ee	 
	Combining \eqref{eq:obj-sub-converge} with \eqref{eq:objection converge}, we have $H(\hx_{\infty})= \bar{H}$. Further, by the arbitrariness of $\hx_{\infty}$, $H(\hx(t))$ is constant on $\cC$. We have item (ii). 
	
	Let $t_1=0$, $t_2\to+\infty$ in \eqref{eq:descent}. Then by \eqref{eq:objection converge},  we have 
	\[
	\lambda^{-1}\int_{0}^{\infty}{\dist^2(0,\partial H(\hx(t)))}dt\leq H(\hx(0))- \bar{H}<+\infty. 
	\]	
	Following \cite[Proposition 6.5.1]{Aubin1984}, if $\hx$ is  absolutely continuous, the proof of \cite[corollary 3]{Duchi2018} is also true. Therefore,  there must exist an increasing sequence $t_k$ identical to the one previously mentioned such that  $\dist(0,\partial H(\hx(t_k)))\to 0$. Then by the outer semicontinuity of $\partial H$, we have $\dist(0,\partial H(x_{\infty}))= 0$, i.e., $x_{\infty}\in\crit H$, hence item (iii) is derived.

    Due to the boundedness of  $\hx$,  $\cC$ is nonempty and bounded. Using the same proof  as \cite[Lemma 3.3(vii)]{Radu2020},  it follows that $\cC$ is closed. Thus,  $\cC$ is compact.  From the definition of cluster point,  we obtain  $\dist((\hx(t)),\cC)\to 0$, which completes the proof of  item (iv). 
\end{proof} 
\begin{remark}
Theorem \ref{th:weak-convergence} (iii) and (iv) show that any cluster point of $\hx(t)$ is a critical point of $H$, and indicate  the existence of the cluster point of a bounded trajectory. In next subsection, we will  establish stronger convergence results under the KL property.	
\end{remark}

\subsection{Global convergence under KL property}
To continue  the global convergence analysis, let us first provide the definition of the  KL property. 
\begin{definition}\label{def:kl}
	A locally Lipschitz continuous function $f:\R^n \to\R$ is said to satisfy the \emph{Kurdyka-\L{}ojasiewicz (KL) property} at $\bar{x}$ if there exist a constant $\eta\in(0,+\infty]$, a neighborhood $U$ of $\bar{x}$ and a continuous concave function $\varphi:[0,\eta)\to[0,+\infty)$ such that
	\begin{itemize}
		\item[(i)] $\varphi(0)=0$;
		\item[(ii)] $\varphi$ is continuously differentiable and $\varphi'>0$ on $(0,\eta)$;
		\item[(iii)] for all $x\in U\cap \{x:f(\bar{x})<f(x)<f(\bar{x})+\eta\}$, the following KL inequality holds
		\[\varphi'(f(x)-f(\bar{x}))\cdot\dist(0,\partial f(x))\geq 1.\]
	\end{itemize}
\end{definition}
A locally Lipschitz function $f$ is said to be a \emph{KL function}, if it satisfies the KL property at every point $x\in\R^n$. In particular, if $\varphi(s)=\sigma s^{1-\theta}$ with $\sigma>0$  and $\theta\in(0,1)$,  we say that  $f$  is a KL function with \emph{\L{}ojasiewicz exponent} $\theta$. 	

\begin{remark}
   The concept of the KL property was originally introduced in seminal works such as \cite{1963Une,1998Kurdyka}, and has  been further developed and extended by subsequent studies, including \cite{2007Bolte,AB2009, ABRS2010}. The KL property is applicable to a wide range of functions, such as semialgebraic functions, real subanalytic functions, and functions that can be defined within an $o$-minimal structure. In the context of the problems considered in this paper, the function $H$ can possess the KL property if both $f$ and $h$ are locally Lipschitz semialgebraic functions. For detailed properties and examples of KL functions, we refer interested readers to \cite{AB2009, ABRS2010, BST2014, DDKL2020}.	 
\end{remark}


We are now ready to establish the global  convergence of $\hx(t)$ generated by the  differential inclusion \eqref{eq:inclusion} in the context of KL property. 
\begin{theorem}\label{th:global-conv}
Under Assumptions \ref{ass:det} and \ref{ass:chain-rule},  suppose that $H$ is a  KL  function and $\hx$ is a bounded  trajectory generated by \eqref{eq:inclusion}. Then  $\hx(t)$ converges to a critical point of $H$.
\end{theorem}
\begin{proof}
As shown in Theorem \ref{th:weak-convergence}, $H(\hx(t))$ converges to a constant, denoted by $\bar H$.  
If there exists $l_0>0$ such that $H(\hx(l_0))=\bar{H}$, it must hold that $H(\hx(t))=\bar{H}$ for any $t\geq l_0$, since $H(\hx(t))$ is nonincreasing by Lemma \ref{lem:descent}. Then by integrating both sides of  \eqref{eq:differential-dist} over $[l_0,+\infty)$, we have 
\[
\lambda^{-1}\int_{l_0}^{\infty}{\dist^2(0,\partial H(\hx(t)))}dt\leq 0. 
\]
Therefore, $\dist(0,\partial H(\hx(t)))=0$ for a.e. $t\geq l_0$. That means, $\hx(t)$ is the critical point of $H$ for a.e. $t\geq l_0$.

Otherwise, suppose that $H(\hx(t))> \bar{H}$ for any $t> 0$. By \eqref{eq:objection converge}, for any $\eta>0$, there exists $l_1>0$ such that 
\[H(\hx(t))< \bar{H}+\eta, \quad \forall t\geq l_1.\]
Theorem \ref{th:weak-convergence} (iv) implies that for any $\varepsilon>0$, there exists $l_2>0$ such that 
\[\dist(\hx(t),\cC)<\varepsilon, \quad \forall t\geq l_2.\]
Let $T_0:=\max\{l_1, l_2\}$. Then it holds that \be\label{inter}
\dist(\hx(t),\cC)<\varepsilon \text{ and } \bar{H}<H(\hx(t))<\bar{H}+\eta, \quad \forall t\geq T_0.
\ee
Since $H(\hx)$ is  a KL function, then from  the uniformized KL property (\cite[Lemma 6]{BST2014}),
there exists a continuously differentiable and  concave function $\varphi$ such  that
\be\label{eq:kl-phi}
\varphi'(H(\hx)-\bar{H})\cdot \dist(0,\partial H(\hx))\geq 1, \, \forall t\geq T_0.
\ee
Meanwhile,  $H(\hx)$ admits the chain rule, which, together with \eqref{eq:chain-rule-application}, leads  to 
\bee
\begin{aligned}
\frac{d}{dt}\varphi(H(\hx(t))-\bar{H})&=\varphi'(H(\hx(t))-\bar{H})\cdot\langle \partial H(\hx(t)), \dot{\hx}(t)\rangle\\ 
&= \varphi'(H(\hx(t))-\bar{H})\cdot(-\lambda\|\dot{\hx}(t)\|^2).
\end{aligned}
\eee
Combining the above equality with  \eqref{eq:dist-dot-x} and \eqref{eq:kl-phi},  we have 
\be\label{eq:varphi-dotx}
\frac{d}{dt}\varphi(H(\hx(t))-\bar{H})\leq  \frac{-\lambda\|\dot{\hx}(t)\|^2}{\dist(0,\partial H(\hx(t)))}\leq -\|\dot{\hx}(t)\|.
\ee
Since $H(\hx(t))\to \bar{H}$ (from \eqref{eq:objection converge}), it holds from by the continuity of $\varphi$ that  $\varphi(H(\hx(t)))\to \varphi(\bar{H})$.
Then, by integrating \eqref{eq:varphi-dotx} on $[T_0,+\infty)$, we obtain
\bee\label{eq:finite length}
\int_{T_0}^{\infty}{\|\dot{\hx}(t)\|} dt<+\infty,
\eee
which implies that for any $\epsilon>0$, there exists $T\ge T_0$ such that for every $t_2\ge t_1\geq T$, 
\[\|\hx(t_2)-\hx(t_1)\|\leq \int_{t_1}^{t_2}{\|\dot{\hx}(t)\|} dt<\epsilon.\]
Hence, according to Cauchy's criterion for convergence,  the limit of $\hx(t)$ exists,  
and by Theorem \ref{th:weak-convergence} (iii), $\hx(t)$ converges globally to a critical point  of $H(\hx)$. 
\end{proof} 

The convergence rates of $\hx(t)$ generated by \eqref{eq:inclusion} are provided in the following theorem, which is  associated with \L{}ojasiewicz  exponent $\theta$.

\begin{theorem}\label{th:conv-rate}
 Under Assumptions \ref{ass:det} and \ref{ass:chain-rule}, suppose that  $H$ is a KL function with  \L{}ojasiewicz exponent $\theta$ and $\hx$ is a bounded trajectory of \eqref{eq:inclusion}. Then the following statements hold true:
\begin{itemize}
	\item[(i)] if $\theta\in(0,1/2]$, there exist  constants $a_1, a_2, b_1>0$ and $T_1>0$ such that 
	\[
	H(\hx(t))-\bar{H}\leq a_1\exp{(-b_1 t)} \ \text{ and }\  \|\hx(t)-x_{\infty}\|\leq a_2\exp{(-b_1(1-\theta) t)},  \ \forall t\geq T_1\,;
	\]
    \item[(ii)] if $\theta\in(1/2,1)$, there exist  constants $c_1, c_2>0$ and  $T_2>0$ such that 
    \[
    H(\hx(t))-\bar{H}\leq c_1 t^{\frac{1}{1-2\theta}} \  \text{ and } \ 
    \|\hx(t)-x_{\infty}\|\leq c_2t^{\frac{1-\theta}{1-2\theta}}, \ \forall t\geq T_2.
    \]	
\end{itemize}
Here, $x_{\infty}$ is the limit point of $\hx(t)$ and $H(\hx(t))$  converges to $ \bar{H}$. 
\end{theorem}
\begin{proof}
Without loss of generality, we assume that $H(\hx(t))>\bar H$ for all $t\ge0$. With a little abuse of notation, we still use the same symbol $T_0$ such that \eqref{inter} is satisfied. 
Since $H(\hx)$ is a KL function with the \L{}ojasiewicz exponent $\theta$, substituting $\varphi(s)=\sigma s^{1-\theta}$ into \eqref{eq:kl-phi},  we attain that for any $t\geq T_0$,
\be\label{eq:rates1}
\sigma (1-\theta)\frac{\dist(0,\partial H(\hx(t)))}{(H(\hx(t))-\bar{H})^{\theta}}\geq 1,
\ee
which, together with \eqref{eq:differential-dist}, yields 
\be\label{eq:kl-H}
\frac{d}{dt}H(\hx(t))\leq \frac{-\lambda^{-1}}{\sigma^2(1-\theta)^2}(H(\hx(t))-\bar{H})^{2\theta}.
\ee
As $H(\hx(t))\to \bar{H}$, there exists $T_1\geq T_0$ such that $H(\hx(t))- \bar{H}<1$ for any $t\geq T_1$. Hence, if $\theta\in(0,1/2)$, it follows from  \eqref{eq:rates1} that
\[
\sigma \frac{\dist(0,\partial H(\hx(t)))}{(H(\hx(t))-\bar{H})^{1/2}}\geq 1,
\]
which indicates that $H(\hx(t))$ satisfies  the KL inequality  with exponent $1/2$. Thus we only need to prove  item (i)  for  $\theta=1/2$. When $\theta=1/2$, it indicates from \eqref{eq:kl-H} that
\bee
\frac{d}{dt}(H(\hx(t))-\bar{H})\leq \frac{-4}{\sigma^2\lambda}(H(\hx(t))-\bar{H}).
\eee
By  rearranging the above inequality and noticing $H(\hx(t))> \bar H$, we obtain 
\bee
\frac{d}{dt}\log(H(\hx(t))-\bar{H})\leq \frac{-4}{\sigma^2\lambda}.
\eee
Integrating the above inequality from $T_1$ to $t$ with $t\geq T_1$, it follows that 
\bee
\log\left(\frac{H(\hx(t))-\bar{H}}{H(\hx(T_1))-\bar{H}}\right)\leq \frac{-4}{\sigma^2\lambda} (t-T_1),
\eee
which derives
\[
H(\hx(t))-\bar{H}\leq (H(\hx(T_1))-\bar{H})\cdot\exp\left(\frac{-4}{\sigma^2\lambda}(t-T_1)\right).
\]
Therefore, the convergence rate of $H(\hx)$ of item (i) with $a_1:= (H(\hx(T_1))-\bar{H})\cdot\exp(\frac{4 T_1}{\sigma^2\lambda})$ and  $b_1:=\frac{4}{\sigma^2\lambda}$ is proved.
   
If $\theta\in(1/2,1)$,  
from \eqref{eq:kl-H} it holds that
\bee
(H(\hx(t))-\bar{H})^{-2\theta}\frac{d}{dt}(H(\hx(t))-\bar{H})= \frac{1}{1-2\theta}\frac{d}{dt}(H(\hx(t))-\bar{H})^{1-2\theta}
\leq \frac{-1}{\sigma^2(1-\theta)^2\lambda}. 
\eee 
Integrating both sides of the above inequality over $[T_0,t]$ for $t\geq T_0$, then by  $1-2\theta<0$ we obtain 
\bee
(H(\hx(t))-\bar{H})^{1-2\theta}\geq (H(\hx(T_0))-\bar{H})^{1-2\theta}+ \frac{2\theta-1}{\sigma^2(1-\theta)^2\lambda}(t-T_0)\geq \frac{2\theta-1}{\sigma^2(1-\theta)^2\lambda}(t-T_0).
\eee  
Thus, there is a positive constant $c_1>0$ such that for $t\geq T_0$,
\[
H(\hx(t))-\bar{H}\leq \left(\frac{2\theta-1}{\sigma^2(1-\theta)^2\lambda}(t-T_0)\right)^{\frac{1}{1-2\theta}}\leq c_1t^{\frac{1}{1-2\theta}}.
\]
We thus derive the convergence rate of $H(\hx)$ in item (ii) with $T_2:=T_0$.

From \eqref{eq:varphi-dotx}, for any $t_2\ge t_1\ge T_0$, it holds that 
\bee\label{eq:length}
\|\hx(t_2)-\hx(t_1)\|=\|\int_{t_1}^{t_2}{\dot{\hx}(t)} dt \|\leq\int_{t_1}^{t_2}{\|\dot{\hx}(t)\|} dt\leq \varphi(H(\hx(t_1))-\bar{H})-\varphi(H(\hx(t_2))-\bar{H}).
\eee
Letting $t_2\to\infty$, by $\hx(t_2)\to x_{\infty}$ we have 
\[
\|\hx(t_1)-x_{\infty}\|\leq \sigma (H(\hx(t_1))-\bar{H})^{1-\theta} \text{ for } t_1\geq T_0.
\]
Hence, the convergence rates of $\hx(t)$ can be obtained accordingly from the above relation and the  convergence rates of $H(\hx(t))$. 
\end{proof}
\section{LP-SADMM algorithm and convergence analysis}\label{sec:sto-ADMM}
In this section, we focus on problem \eqref{eq:finite-p}, where   $h$ is  $\varrho$-weakly convex. To tackle the potential difficulties caused by the weak convexity in theoretical analysis, we  draw into  the utilization of smooth Moreau envelop of $h$.  Since $N$ can be very large in many scenarios, computing the full gradient $\nabla f$ at any given point becomes challenging. To cope with this issue, we propose to use a stochastic gradient estimator $\Sto f$ to approximate the full gradient. This leads to the development of a linearized proximal stochastic ADMM method (LP-SADMM). We then derive the continuous-time stochastic differential equation of the proposed algorithm  and explore its convergence properties.

\subsection{LP-SADMM and continuous-time system}\label{subsec:lp-sadmm-system}

The main algorithmic framework of LP-SADMM is presented in Algorithm \ref{alg:sto-Modified ADMM}. 
\begin{algorithm}
\caption{ LP-SADMM }\label{alg:sto-Modified ADMM}
\begin{algorithmic}[1]
\Require Initial point $(x^0,z^0,u^0) \in \R^n\times\R^m\times\R^m$,  parameters $\eta,\rho>0$, $\tau>\rho\|A^TA\|+1/\eta$ and $0<\mu<1/\varrho$.
\For{ $k=0,1,2,\ldots$ } 
\State Update $x^{k},z^{k},u^{k}$ as follows:
\begin{subnumcases}{\label{eq:iter-sto}}
				x^{k+1}=x^k-\frac{1}{\tau}\left(\Sto f(x^k)+\rho A^T(Ax^k-z^k+\frac{1}{\rho}u^k)\right),\label{eq:iter1-sto}\\
				\notag\\
				z^{k+1}=\argmin_{z\in\R^m}\left\lbrace h_{\mu}(z)+\langle u^k, Ax^{k+1}-z\rangle+\frac{\rho}{2}\|Ax^{k+1}-z\|^2\right\rbrace,\label{eq:iter2-sto}\\
				\notag\\
				u^{k+1}=u^k+Ax^{k+1}-z^{k+1}\label{eq:iter3-sto}.
\end{subnumcases}
\State  Set $k \Leftarrow k+1$.
\EndFor 
\end{algorithmic}
\end{algorithm}
 
In the remainder of this subsection, we aim  to derive  the continuous-time system of Algorithm \ref{alg:sto-Modified ADMM}, which is a first-order stochastic differential equation.   
To simplify the  notation, we use $\xi^k:=\Sto f(x^k)-\nabla f(x^k)$ to denote the gradient noise at $x^k$.
We now state some  essential assumptions  regarding  the objective function of \eqref{eq:finite-p} and stochastic gradient noise below.
\begin{assumption}\label{ass:acc-setting}
	 Assumption \ref{ass:det} holds. Furthermore,  suppose  that 
	\begin{itemize}
		\item[(i)] Function $f$ is  $L_f$-smooth, i.e., $f$ is  continuously differentiable with $L_f$-Lipschitz continuous gradients, that is,
		\[\|\nabla f(x)-\nabla f(y)\|\leq L_f\|x-y\|, \quad \forall x, y\in\R^n.\]		
		\item[(ii)] Function $h$ is $\varrho$-weakly convex, and is $L_h$-Lipschitz continuous, i.e.
  \[|h(x)-h(y)|\leq L_h\|x-y\|, \quad \forall x, y\in\R^n.\]
	\end{itemize}	
\end{assumption} 

\begin{remark}\label{re:moreau} Some comments  are given as follows.
	\begin{itemize}
		\item[(i)]  Under  Assumption \ref{ass:det}(i), as discussed in Subsection \ref{sec:moreau and weak convex}, $h_{\mu}$ is lower bounded. Moreover, it is  convex  and  $(\max\{\frac{1}{\mu}, \frac{\varrho}{1-\varrho\mu}\})$-smooth following   Assumption \ref{ass:acc-setting}(ii) and  the setting  $\mu < 1/\varrho$ in  Algorithm \ref{alg:sto-Modified ADMM}.
		\item[(ii)] Under Assumption \ref{ass:acc-setting}, it is easy to check  that  $H_{\mu}$ is $L$-smooth, where $H_\mu(x):=f(x)+h_{\mu}(Ax)$ and  $L:=L_f+\max\{\frac{1}{\mu}, \frac{\varrho}{1-\varrho\mu}\}\cdot\|A\|^2$. Besides, Assumption \ref{ass:det}(i) ensures the lower boundedness of  $H_{\mu}$. 
		\item[(iii)] According to  \cite[Subsection 3.2]{Axel2021}, Assumption \ref{ass:det}(iii) and Assumption \ref{ass:acc-setting}(i)-(ii) are   necessary  for establishing the relationship between  $\crit H$ and $\crit H_{\mu}$ as shown in Lemma \ref{lem:critical}.
	\end{itemize}
\end{remark}

Motivated by recent developments in stochastic algorithms such as  \cite{Mandt2015ContinuousTime, stephan2016, Li2017icml, DBLPshibin}, we make the following assumption on the random noise $\xi^k$. 
\begin{assumption}\label{ass:first-sto}
Assume that  for all $k\ge0$, the gradient noise  $\xi^k\overset{i.i.d.}{\sim} \cN(0,1)$. 
\end{assumption}
\begin{remark}\label{re:first-sto-setting} Some comments on Assumption \ref{ass:first-sto} are provided as follows. 
\begin{enumerate}
\item[(i)] Due to the finite-sum structure of $f$ in \eqref{eq:finite-p}, a popular way to compute the stochastic gradient at given point $x$ is through  $\Sto f(x):=\frac{1}{|\cS|}\sum_{i\in \cS} \nabla f_{i}(x)$, if $f_i, i=1,\cdots,N$ are differentiable. Here, $\cS$ is a set of randomly selected indices drawn uniformly at random from $\{1,\cdots, N\}$ and $|\cS|$ represents the cardinality of set $\cS$. By the (Lindeberg-L\'{e}vy) central limit theorem, the stochastic gradient  $\Sto f(x)$ is  Gaussian distributed with variance $\propto 1/|\cS|$ for  sufficiently large $|\cS|$. 
This ensures that Assumption \ref{ass:first-sto} is reasonable. 
\item[(ii)] As presented in Subsection \ref{sec:brownian},  a standard  Brownian motion $\{W(t), t\geq 0\}$ has the property that $W(t)-W(s)\sim \cN(0,t-s)$ for any $ t>s\geq 0$.  Thus under Assumption \ref{ass:first-sto}, $\xi^k$ can be expressed as
\be\label{eq:Wiener process}
\xi^k=\sqrt{\rho}\left(W\left(t+\frac{1}{\rho}\right)-W(t)\right) \ \text{ with } \  t=\frac{k}{\rho}.
\ee	
\end{enumerate}
\end{remark}

We now proceed to derive a stochastic continuous dynamical system  of  Algorithm \ref{alg:sto-Modified ADMM}. Let $\{x^k\}, \{z^k\}, \{u^k\}$ be generated by Algorithm \ref{alg:sto-Modified ADMM}. By substituting $\xi^k=\Sto f(x^k)-\nabla f(x^k)$ into \eqref{eq:iter1-sto}, and combining the optimality condition for \eqref{eq:iter2-sto}: 
\[0=\nabla h_\mu(z^{k+1})-\rho(Ax^{k+1}-z^{k+1}+\frac{1}{\rho}u^k),\]
we obtain
\be\label{eq:sto-opt-f}
0=\tau(x^{k+1}-x^k)+\nabla f(x^k)+A^{T}\nabla  h_\mu(z^{k+1})+\xi^k-\rho A^{T}A(x^{k+1}-x^{k})+\rho A^T(z^{k+1}-z^k).
\ee
Similar to the previous section, we introduce the \emph{Ansatz}:   $x^k\approx\hx(k/\rho)$, $z^k\approx\hz(k/\rho)$ and  $u^k\approx\hu(k/\rho)$ for some continuously differentiable stochastic processes $\hx$, $\hz$ and $\hu$. Put  $k=t/s$ with $s=1/\rho$. When $s$ is sufficiently small, it follows that $\hx(t)\approx x^{t/s}=x^k, \hx(t+s)\approx x^{(t+s)/s}=x^{k+1}, \hz(t)\approx z^{t/s}=z^k, \hz(t+s)\approx z^{(t+s)/s}=z^{k+1}, \hu(t)\approx u^{t/s}=u^k, \hu(t+s)\approx u^{(t+s)/s}=u^{k+1}$.
  Combining \eqref{eq:Wiener process} and \eqref{eq:sto-opt-f} and applying Taylor's theorem  yield 
	\be\label{eq:first-sto-inclusion}
 \begin{aligned}
	0&=\tau(s\dot{\hx}(t)+O(s^2))+\nabla f(\hx(t))+A^{T}\nabla  h_\mu(\hz(t)+O(s))+s^{-1/2}(W(t+s)-W(t))\\
 &\quad\quad-A^TA\dot{\hx}(t)+A^T\dot{\hz}(t).
 \end{aligned}
	\ee
For the term $W(t+s)-W(t)$, the following holds true: 
	\be\label{eq:second-inclusion-w}
	W(t+s)-W(t)=s\dot{W}(t)+O(s^2).
	\ee 
	Then, by inserting \eqref{eq:second-inclusion-w} into \eqref{eq:first-sto-inclusion} and due to $\tau s> \|A^T A\| +s/\eta$,  there exists a constant $\lambda> \|A^TA\|$ such that 
	\be\label{inc-s}
	0=\lambda\dot{\hx}(t)+O(s)+\nabla f(\hx(t))+A^{T}\nabla  h_\mu(\hz(t)+O(s))+s^{1/2}\dot{W}(t)-A^TA\dot{\hx}(t)+A^T\dot{\hz}(t).
	\ee
In an analogous way to  \eqref{eq:condition1} and \eqref{eq:condition2},  it holds that  for any $t\in[0,+\infty)$, $A\hx(t)=\hz(t)$ and $A\dot{\hx}(t)=\dot{\hz}(t)$ as $s\to 0$, which, together with the continuity of $\nabla  h_\mu$,  implies $\nabla  h_\mu(\hz(t)+O(s))\to \nabla  h_\mu(A\hx(t))$.  
For the term $s^{1/2}\dot{W}(t)=s^{1/2}dW/dt$, by \eqref{eq:Wiener process} its  Euler discretization is $s^{1/2}\sqrt{\Delta t}\xi^k/\Delta t$.  
 Since $\Delta t= s$, the Euler discretization of $s^{1/2}\dot{W}(t)$ is the random variable $\xi^k$.  
 Therefore, as $s\to 0$, 
 the term $s^{1/2}\dot{W}(t)$ cannot be neglected,  and we thus consider the following stochastic differential equation 
\be\label{eq:sto-inclusion}
0=\lambda\dot{\hx}(t)+\nabla   H_{\mu}(\hx(t))+\rho^{-1/2}\dot{W}(t), \text{ with } \hx(0)=x^0,
\ee
where $\lambda> \|A^TA\|$.  

Following  \cite{Li2017icml},   we next present an approximation theorem  which indicates that  the stochastic differential equation \eqref{eq:sto-inclusion} is a weak approximation of Algorithm \ref{alg:sto-Modified ADMM} with $k=t\rho$ under certain conditions.   The concept of order $\alpha$ weak approximation  was originally  introduced in \cite[Definition 2]{Li2017icml}. For completeness, we state it here.  

\begin{definition}\label{def:weak approximation}
Let $T>0$ and $\alpha\geq 1$ be integers, and  $\rho>\max\{1,1/T\}$. Let $\cG$ be the set of all continuous functions $g:\R^n\to\R$ satisfying the condition that   there exist integers $\kappa_1,\kappa_2>0$ such that 
\[|g(x)|\leq \kappa_1(1+\|x\|^{2\kappa_2}), \quad\forall x\in \R^d. \]  
Let $\cG^{\alpha}$ denote the  set of $\alpha$-times continuously differentiable functions  which, together with its partial derivatives up to and including order $\alpha$, belong to $\cG$.   
 A continuous-time stochastic process $\{\hx(t):t\in[0,T]\}$ is  said to be an order $\alpha$ weak approximation of a discrete stochastic process $\{x^k: k\geq 0\}$, if for every  $g\in\cG^{\alpha+1}$, 
there exists a constant $C>0$ (independent of $\rho$)  such that 
\[\max_{k=0,1,\cdots, \lfloor \rho T\rfloor}\left|\Exp[g(x^k)]-\Exp[g(\hx(k/\rho))]\right|\leq \frac{C}{\rho^{\alpha}}.\]
\end{definition}

Now, we give the following order $1$ weak approximation theorem.  The proof of this theorem can be referred to \cite[Corollary 10]{Li2017icml}. 
\begin{theorem}\label{th:weak approximation}
Given  integer $T>0$, let $\{(x^k,z^k,u^k): k\geq 0\}$ be the iteration sequence generated by Algorithm \ref{alg:sto-Modified ADMM} with $\rho>\max\{1,1/T\}$.  
Suppose that the following conditions are satisfied: 
\begin{itemize}  
\item[(i)] $H_{\mu}$ is continuously differentiable with Lipschitz continuous gradients;
\item[(ii)] $H_{\mu}$ has weak derivatives (refer to  \cite[Definition 7]{Li2017icml}) up to order $3$, and for any index $|\beta|\leq 3$, there exist integers $\kappa_1,\kappa_2>0$ such that 
\[\|D^{\beta} H_{\mu}(x)\|\leq \kappa_1(1+\|x\|^{2\kappa_2}) \text{ for a.e. } x\in \R^d, \]
where $D^{\beta} H_{\mu}$ represents the order $\beta$ weak derivative of $H_{\mu}$.
\end{itemize}
Then $\{\hx(t):t\in[0,T]\}$ satisfying \eqref{eq:sto-inclusion}  is an order $1$ weak approximation of Algorithm \ref{alg:sto-Modified ADMM}. 
\end{theorem}

The Lipschitz continuity of $\nabla H_{\mu}$ implies that the  linear growth condition  holds: $\|\nabla H_{\mu}(x)\|\leq C(1+\|x\|)$ for some constant $C>0$ and any $x\in \R^n$. Hence according to \cite[Theorem 5.2.1]{Oksendal2003}, Theorem \ref{th:weak approximation} (i) ensures that \eqref{eq:sto-inclusion} admits a unique solution.

Let the stochastic process $\hx$ be a trajectory of \eqref{eq:sto-inclusion}.  We claim that for any twice continuously differentiable function $\phi: \R^n\to \R$,  the following chain rule  holds in the stochastic setting:
\be\label{eq:sto-chain rule for H}
\frac{d\phi(\hx(t))}{dt}=\langle \nabla \phi(\hx(t)), \dot{\hx}(t)\rangle.
\ee
To elaborate this claim, let us  recall It{\^o}'s  formula \eqref{eq:ito-lemma}. It follows that 
\be\label{eq:sto-complex chain rule for H}
\frac{d\phi(\hx(t))}{dt}=\langle \nabla \phi(\hx(t)), \dot{\hx}(t)\rangle+O(1/\rho).
\ee
Then for sufficiently large $\rho$, there must exist  a stochastic process $\tilde{\hx}$, which is very close to $\hx$, such that 
\be\label{eq:sto-chain rule for tilde}
\frac{d\phi(\tilde{\hx}(t))}{dt}=\langle \nabla \phi(\tilde{\hx}(t)), \dot{\tilde{\hx}}(t)\rangle.
\ee
Note that the trajectory $\hx$, which is  an order $1$ weak approximation of $x^k$ generated by Algorithm \ref{alg:sto-Modified ADMM} under the conditions of Theorem \ref{th:weak approximation}, is established by taking the limit  of \eqref{inc-s} as $s\to 0$, where $s=1/\rho$. Then, we can say that \eqref{eq:sto-chain rule for tilde}  is  satisfied by $\hx$ when $\rho$ is sufficiently large,  rather than by  $\tilde{\hx}$. 
We will utilize the chain rule \eqref{eq:sto-chain rule for H} to analyze properties of $H_{\mu}(\hx)$ 
in the remainder of this section, without assuming the twice differentiability of $H_{\mu}$.

\subsection{Convergence analysis}
In this subsection, we first demonstrate that the trajectory $\hx$ of  \eqref{eq:sto-inclusion} converges almost surely  to the critical point of the function $H_\mu$, under the assumptions that the   trajectory $\hx$ is bounded and $H_\mu$ is a KL function with \L{}ojasiewicz exponent. Next, with the aid of the relationship between the critical points of $H$ and $H_{\mu}$ in Lemma \ref{lem:critical} (with proof referred to  \cite[Subsection 3.2]{Axel2021}), we prove that  the function $\bar\hx$, defined by 
\be\label{bar-x}
\bar{\hx}(t):=\hx(t)-A^T(AA^T)^{-1}(A\hx(t)-\prox_{\mu h}(A\hx(t))),
\ee
converges to an approximate critical point of the objective function $H$ when  $\mu$ is sufficiently small. Here, the invertibility of $AA^T$ is assured by Assumption \ref{ass:det}(iii).

\begin{lemma}\label{lem:critical}
	Suppose that  Assumption \ref{ass:acc-setting} holds. For any given $x\in\crit H_{\mu}$, let 
	\be\label{def:critical-H}
	\bar{x}:=x-A^T(AA^T)^{-1}(Ax-\prox_{\mu h}(Ax)).
	\ee
	 Then it holds that
	\[\dist(0,\partial H(\bar{x}))\leq \frac{L_fL_h\mu}{\sqrt{\ewmin(AA^T)}},\]
	where  $\ewmin(AA^T)$ denotes the smallest eigenvalue of $AA^T$. 	
\end{lemma}

Given $\epsilon>0$, we call  a point $x\in \mathbb R^n$  an $\epsilon$-approximate critical point of $H$, if $\dist(0,\partial H(x))\le \epsilon. $ The set of all $\epsilon$-approximate critical point of $H$ is denoted by $\crit_\epsilon H.$  
Motivated by Lemma \ref{lem:critical}, to approach an $\epsilon$-approximate critical point of $H$, it is  sufficient to get close to a critical point of $H_\mu$. More specifically, if we can prove the convergence of trajectory $\hx(t)$ of \eqref{eq:sto-inclusion} towards a critical point of $H_{\mu}$ under certain conditions,  then by applying Lemma \ref{lem:critical} we can show  that 
$\bar\hx(t)$, as defined through \eqref{bar-x}, 
is convergent to an $\epsilon$-approximate  critical point of $H$, provided that   $\mu \le {\epsilon\sqrt{\lambda_{min}(AA^T)}}/{(L_fL_h)}$.  Thus, our goal in the remainder of this  subsection is to show the convergence of $\hx(t)$ to some critical point  of $H_{\mu}$. 

Before proceeding, we first present a descent property of $\Exp[H_{\mu}(\hx)]$   based on stochastic system \eqref{eq:sto-inclusion}.

\begin{lemma}\label{lem:sto-descent}
Suppose that Assumptions \ref{ass:acc-setting} and \ref{ass:first-sto} hold. Let  $\hx$ be a bounded trajectory of   \eqref{eq:sto-inclusion}. For  any  $t_2\ge t_1\geq 0$,  it holds that 
	\be\label{eq:inter-obj}
	\Exp[H_{\mu}(\hx(t_2))]+\lambda\Exp\left[\int_{t_1}^{t_2}{\|\dot{\hx}(t)\|^2}dt\right]=\Exp[H_{\mu}(\hx(t_1))].
	\ee
\end{lemma}
\begin{proof}
By applying \eqref{eq:sto-chain rule for H},  the following relation holds for $t\geq 0$:
\be\label{eq:sto-chain-rule-application}
\begin{aligned}
	\frac{d}{dt}H_{\mu}(\hx(t))&=\langle \nabla H_{\mu}(\hx(t)), \dot{\hx}(t)\rangle 
	= \langle -\lambda \dot{\hx}(t)-\rho^{-1/2}\dot{W}(t), \dot{\hx}(t)\rangle\\
 &= -\lambda\|\dot{\hx}(t)\|^2-\rho^{-1/2}\langle\dot{W}(t), \dot{\hx}(t)\rangle,
\end{aligned}
\ee	
where the second equality is from \eqref{eq:sto-inclusion}. 
Integrating  \eqref{eq:sto-chain-rule-application} from $t_1$ to $t_2$ yields
\be\label{eq:par1}
H_{\mu}(\hx(t_2))-H_{\mu}(\hx(t_1))=-\lambda \int_{t_1}^{t_2}{\|\dot{\hx}(t)\|^2}dt-\rho^{-1/2}\int_{t_1}^{t_2}{\langle\dot{W}(t), \dot{\hx}(t)\rangle}dt.
\ee
As the trajectory $\hx$  is bounded,  $\dot{\hx}\in \mathbb L_e^2(0, t)$ for any $t> 0$.  From Subsection \ref{sec:brownian}, we know that $\int_{0}^{t}{\langle\dot{W}(s), \dot{\hx}(s)\rangle}ds$ is an It{\^o}'s integral of $\dot{\hx}$ on $[0,t]$. 
It thus indicates 
\[
\Exp\left[\int_{0}^{t}{\langle\dot{W}(s), \dot{\hx}(s)\rangle}ds\right]=0 \text{ for } t> 0.
\]
By taking expectation on both sides of \eqref{eq:par1}  and  using the above equality, we obtain \eqref{eq:inter-obj}. The proof is completed.
\end{proof}
The trajectory $\hx$ of \eqref{eq:sto-inclusion} is a stochastic process defined on the probability space $(\Omega,\cF,\Prob)$. 
We now define the set consisting  of  all cluster points of the stochastic process $\hx(t)$ as follows: \[\cC:=\{x_\infty: \exists \text{ an increasing subsequence } \{t_k\} \text{ such that } \hx(t_k)\to x_\infty \text{  a.s.} \text{  as } k\to+\infty\}.\]    
Based on Lemma \ref{lem:sto-descent},  we obtain the  following theorem.

\begin{theorem}\label{th:sto-weak-convergence}
Suppose that Assumptions  \ref{ass:acc-setting} and \ref{ass:first-sto} hold and let $\hx$ be a bounded trajectory generated by \eqref{eq:sto-inclusion}. Then  we have:
\begin{itemize}
	\item[(i)] $\int_{0}^{\infty}{\| \dot{\hx}(t)\|^2}dt<+\infty$ a.s.;
 \item[(ii)] the set $\cC$ is nonempty, almost surely compact and \[\lim_{t\to\infty}\dist((\hx(t)),\cC)= 0\,;\]
\item[(iii)] $H_{\mu}$ has constant expectation over $\cC$;
\item[(iv)] $\cC\subseteq\crit H_{\mu}$ holds almost surely.	
 \end{itemize}
\end{theorem} 
\begin{proof}
As presented in Remark \ref{re:moreau}(ii),   $\Exp[H_{\mu}(\hx(t))]$ is lower bounded under Assumption \ref{ass:acc-setting}. Moreover, it follows from Lemma \ref{lem:sto-descent} that   $\Exp[H_{\mu}(\hx(t))]$ is nonincreasing. Thus there exists a finite value $\bar{H}_{\mu}$ such that  
\be\label{eq:sto-conve-obj}
\lim_{t\to\infty}\Exp[H_{\mu}(\hx(t))]=\bar{H}_{\mu}.
\ee
Since $\dot{\hx}\in \mathbb L_e^2(t_1,t_2)$  for any  $t_2> t_1\ge 0$, we can interchange the order of expectation and integration in  \eqref{eq:inter-obj}, deriving
\[
\lambda\int_{t_1}^{t_2}{\Exp\left[\|\dot{\hx}(t)\|^2\right]}dt=\Exp[H_{\mu}(\hx(t_1))]-\Exp[H_{\mu}(\hx(t_2))].
\]
Taking $t_1=0$ and $t_2\to +\infty$  and using  \eqref{eq:sto-conve-obj}, we have 
\be\label{eq:sto-finite-2}
\int_{0}^{\infty}{\Exp\left[\| \dot{\hx}(t)\|^2\right]}dt<+\infty,\ee
which implies  item (i).  

If $\hx$ is bounded, by the definition of $\cC$ and the similar approach to prove Theorem \ref{th:weak-convergence}(iv), we obtain item (ii). 
For any $x_{\infty}\in\cC$, by the definition of a cluster point, there exists an increasing sequence $\{t_k\}$ such that $\hx(t_k)\to x_{\infty}$ almost surely.  According to the continuity of $H_{\mu}$, it follows that $ H_{\mu}(\hx(t_k)) \to H_{\mu}(x_{\infty})$ almost surely as $t_k\to +\infty$. Furthermore, by the boundedness of $\hx$ and Lebesgue's dominated convergence theorem, we have
$
\Exp[H_{\mu}(\hx(t_k))]\to \Exp[H_{\mu}(x_{\infty})],
$
which, together with \eqref{eq:sto-conve-obj},  yields that  $\Exp[H_{\mu}(x_{\infty})]=\bar{H}_{\mu}$ for any $x_{\infty}\in\cC$. Hence, $H_{\mu}(\hx(t))$ has constant expectation over $\cC$ which derives 
item (iii).

We next show item (iv), i.e., $x_{\infty}\in\crit H_{\mu}$ almost surely  for any $x_{\infty}\in\cC$. Returning to \eqref{eq:sto-chain-rule-application}, we have
\bee
\begin{aligned}
	\frac{d}{dt}H_{\mu}(\hx(t))&=  \langle \nabla H_{\mu}(\hx(t)), \dot{\hx}(t)\rangle  =  \langle -\lambda \dot{\hx}(t)-\rho^{-1/2}\dot{W}(t), \dot{\hx}(t)\rangle\\
	&= 
	\frac{1}{\lambda}\langle -\lambda \dot{\hx}(t)-\rho^{-1/2}\dot{W}(t), \lambda\dot{\hx}(t)+\rho^{-1/2}\dot{W}(t)\rangle\\
 &\quad\quad+\frac{1}{\lambda}\langle \lambda \dot{\hx}(t)+\rho^{-1/2}\dot{W}(t), \rho^{-1/2}\dot{W}(t)\rangle\\
	&=  -\frac{1}{\lambda}\|\lambda \dot{\hx}(t)+\rho^{-1/2}\dot{W}(t)\|^2+\frac{1}{\lambda}\langle \lambda \dot{\hx}(t)+\rho^{-1/2}\dot{W}(t), \rho^{-1/2}\dot{W}(t)\rangle.
\end{aligned}
\eee
By integrating the above inequality over $[0,t]$, we obtain
\be\label{eq:first-sto-weak1}
\begin{aligned}
&\int_{0}^{t}{\|\lambda \dot{\hx}(s)+\rho^{-1/2}\dot{W}(s)\|^2}ds\\
&=\lambda(H_{\mu}(\hx(0))-H_{\mu}(\hx(t)))+ \int_{0}^{t}{\langle \lambda\dot{\hx}(s)+\rho^{-1/2}\dot{W}(s),\rho^{-1/2} \dot{W}(s)\rangle} ds. 
\end{aligned}
\ee
The boundedness of $\hx$ and the Lipschitz continuity of $\nabla H_{\mu}$  suggest that $\nabla H_{\mu}(\hx)$ is bounded,  and from  \eqref{eq:sto-inclusion},  $\lambda    \dot{\hx}(t)+\rho^{-1/2}\dot{W}(t)$ is bounded  for any  $t> 0$. Hence, $\lambda    \dot{\hx}+\rho^{-1/2}\dot{W}\in \mathbb L_e^2(0, t)$ for any $t>0$. Then according to Subsection \ref{sec:brownian},  $\int_{0}^{t}{\langle \lambda\dot{\hx}(s)+\rho^{-1/2}\dot{W}(s), \dot{W}(s)\rangle} ds$ is an It{\^o}'s integral 
satisfying
\[
\Exp\left[\int_{0}^{t}{\langle \lambda\dot{\hx}(s)+\rho^{-1/2}\dot{W}(s), \dot{W}(s)\rangle} ds\right]=0
\] 
which implies from  \eqref{eq:first-sto-weak1} that 
\bee
\Exp\left[\int_{0}^{t}{\|\lambda \dot{\hx}(s)+\rho^{-1/2}\dot{W}(s)\|^2}ds\right]=\lambda(H_{\mu}(\hx(0))-\Exp[H_{\mu}(\hx(t))])<+\infty.
\eee
Since $\|\lambda \dot{\hx}(s)+\rho^{-1/2}\dot{W}(s)\|^2$ is nonnegative and integrable, we can interchange the order of integration and expectation, obtaining
\be\label{eq:sto-part1}
 \int_{0}^{t}{\Exp\left[\|\lambda \dot{\hx}(s)+\rho^{-1/2}\dot{W}(s)\|^2\right]}ds=\lambda(H_{\mu}(\hx(0))-\Exp[H_{\mu}(\hx(t))]).
\ee
Taking $t\to \infty$ and using  \eqref{eq:sto-conve-obj}, it follows that  
\[
\int_{0}^{\infty}{\Exp\left[\|\lambda \dot{\hx}(s)+\rho^{-1/2}\dot{W}(s)\|^2\right]}ds<\infty,
\]  
which indicates that 
\[\int_{0}^{\infty}{\|\lambda \dot{\hx}(s)+\rho^{-1/2}\dot{W}(s)\|^2}ds<\infty \text{ a.s.}
\] 
By the same line as to prove Theorem \ref{th:weak-convergence}(iii), we attain 
\be\label{eq:sto-gradient-conv}
\|\lambda \dot{\hx}(t_k)+\rho^{-1/2}\dot{W}(t_k)\|\to 0  \text{ a.s.}
\ee
Combining $\hx(t_k)\to x_{\infty}$ almost surely,  $\lambda \dot{\hx}(t_k)+\rho^{-1/2}\dot{W}(t_k)=-\nabla H_{\mu}(\hx(t_k))$ and \eqref{eq:sto-gradient-conv} together, and using the continuity  of $\nabla H_{\mu}$, we obtain  $\dist(0,\nabla H_{\mu}(x_{\infty}))= 0$ almost surely.  Hence, item (iv) is proved.
\end{proof}
Using the KL property and the relationship between $\crit H$ and $\crit H_{\mu}$, we next demonstrate that   $\bar{\hx}(t)$ converges almost surely to an $\epsilon$-approximate  critical point of the objective function $H$. 

\begin{theorem}\label{th:sto-global}
 Suppose that  Assumptions  \ref{ass:acc-setting} and \ref{ass:first-sto} hold,  $H_{\mu}$ is a KL function with \L{}ojasiewicz exponent $\theta$, and $\hx$ is a bounded  trajectory of \eqref{eq:sto-inclusion}.   Then, the following statements hold true: 
\begin{itemize}
	\item[(i)] $\int_{T_0}^{+\infty}{\|\dot{\hx}(t)\|}dt<+\infty$ almost surely for a time $T_0> 0$;
	\item[(ii)] $\hx(t)$ converges almost surely to some critical point of $H_\mu$;
	\item[(iii)]  $\bar{\hx}(t)$, defined through \eqref{bar-x},  converges almost surely to an $\epsilon$-approximate critical point of $H$, if the parameter $\mu$ satisfies $\mu \le {\epsilon\sqrt{\lambda_{min}(AA^T)}}/{(L_fL_h)}$. 
\end{itemize}
\end{theorem}
\begin{proof}
	Since $H_{\mu} $ is a KL function with \L{}ojasiewicz exponent $\theta$, by \cite[Lemma 4.5]{DTLDS2021}, there exist  $T_0>0$ and a continuous concave function $\varphi_0(s)=\sigma_0 s^{1-\theta}$ with $\theta\in(0,1)$ such that 
	\be\label{eq:kl-sto-0}
	\varphi_0'(\Exp[H_{\mu}(\hx(t))]-\bar{H}_{\mu,t})\Exp[\|\nabla H_{\mu}(\hx)\|]\geq 1, \ \forall t \geq T_0,
	\ee
	where $\bar{H}_{\mu,t}$ is a nondecreasing sequence converging to $\bar{H}_{\mu}$. 
	Further, by \eqref{eq:sto-part1} we have 
	\be\label{eq:diff-obj1}
	\begin{aligned}
	\frac{d}{dt}(\Exp[H_{\mu}(\hx(t))]-\bar{H}_{\mu,t})&=\frac{d}{dt}(\Exp[H_{\mu}(\hx(t))]-\bar{H}_{\mu})+\frac{d}{dt}(\bar{H}_{\mu}-\bar{H}_{\mu,t})\\
 &\leq \frac{d}{dt}(\Exp[H_{\mu}(\hx(t))]-\bar{H}_{\mu})\\
	&=-\lambda^{-1}\Exp[\|\lambda \dot{\hx}(t)+\rho^{-1/2}\dot{W}(t)\|^2].
	\end{aligned}
	\ee
	Let $t_1=0, t_2=t$ in \eqref{eq:inter-obj}.  Then the following relation holds:
	\[
	\frac{d}{dt}(\Exp[H_{\mu}(\hx(t))])=-\lambda\Exp[\| \dot{\hx}(t)\|^2],
	\]
which, together with \eqref{eq:diff-obj1},  indicates
	 \be\label{eq:diff-obj}
	\frac{d}{dt}(\Exp[H_{\mu}(\hx(t))]-\bar{H}_{\mu,t})\leq -\lambda^{-1}\Exp[\|\lambda \dot{\hx}(t)+\rho^{-1/2}\dot{W}(t)\|^2]=-\lambda\Exp[\| \dot{\hx}(t)\|^2].
	\ee
	 Taking the total time derivative of $\varphi_0$, we obtain  
	\be\label{eq:sto-varphi-dotx}
	\begin{aligned}
		&\frac{d}{dt}\varphi_0(\Exp[H_{\mu}(\hx(t))]-\bar{H}_{\mu,t})= \varphi_0'(\Exp[H_{\mu}(\hx(t))]-\bar{H}_{\mu,t})\cdot	\frac{d}{dt}(\Exp[H_{\mu}(\hx(t))]-\bar{H}_{\mu,t})\\
		\leq\ &  \frac{-\lambda^{-1}\Exp[\|\lambda \dot{\hx}(t)+\rho^{-1/2}\dot{W}(t)\|^2]}{\Exp[\|\nabla H_{\mu}(\hx)\|]} 
		\leq   \frac{-\lambda^{-1}\left(\Exp[\|\lambda \dot{\hx}(t)+\rho^{-1/2}\dot{W}(t)\|]\right)^2}{\Exp[\|\nabla H_{\mu}(\hx)\|]}\\
		\leq\ & -\lambda^{-1}\Exp[\|\lambda \dot{\hx}(t)+\rho^{-1/2}\dot{W}(t)\|],
	\end{aligned}
	\ee
   where the first inequality is from \eqref{eq:kl-sto-0} and  \eqref{eq:diff-obj}, the second and third inequalities are obtained by Cauchy-Schwarz inequality and $\lambda \dot{\hx}(t)+\rho^{-1/2}\dot{W}(t)=-\nabla H_{\mu}(\hx)$, respectively.
   By inserting \eqref{eq:diff-obj} into \eqref{eq:sto-varphi-dotx}, it also yields that  
\bee
	\frac{d}{dt}\varphi_0(\Exp[H_{\mu}(\hx(t))]-\bar{H}_{\mu,t})
	\leq\frac{-\lambda\Exp[\| \dot{\hx}(t)\|^2]}{\Exp[\|\nabla H_{\mu}(\hx)\|]}.
\eee	
Rearranging the above inequality, for any $t\geq T_0$ we obtain
\be\label{eq:sto-glo1}
\begin{aligned}
\sqrt{\lambda\Exp[\| \dot{\hx}(t)\|^2]}\leq & \ \sqrt{\Exp[\|\nabla H_{\mu}(\hx)\|]\cdot-\frac{d}{dt}\varphi_0(\Exp[H_{\mu}(\hx(t))]-\bar{H}_{\mu,t})}\\
\leq & \ \frac{1}{2}\Exp[\|\nabla H_{\mu}(\hx)\|]-\frac{1}{2}\frac{d}{dt}\varphi_0(\Exp[H_{\mu}(\hx(t))]-\bar{H}_{\mu,t})\\
\leq & \ \frac{1}{2}\Exp[\|\lambda \dot{\hx}(t)+\rho^{-1/2}\dot{W}(t)\|]-\frac{1}{2}\frac{d}{dt}\varphi_0(\Exp[H_{\mu}(\hx(t))]-\bar{H}_{\mu,t})\\
\leq &\ -\frac{1+\lambda}{2}\frac{d}{dt}\varphi_0(\Exp[H_{\mu}(\hx(t))]-\bar{H}_{\mu,t}), 
\end{aligned}
\ee
where the second inequality is obtained by $2\sqrt{ab}\leq a+b$ for any $a,b\ge0$, the third and the last inequalities are  deduced from $\lambda \dot{\hx}(t)+\rho^{-1/2}\dot{W}(t)=-\nabla H_{\mu}(\hx(t))$ and  \eqref{eq:sto-varphi-dotx}, respectively. Applying Cauchy-Schwarz inequality and \eqref{eq:sto-glo1}, 
we have 
\be\label{eq:sto-varphi-dot}
\begin{aligned}
\int_{T_0}^{t}{\Exp[\|\dot{\hx}(s)\|]}ds&\leq \int_{T_0}^{t}{\sqrt{\Exp[\|\dot{\hx}(s)\|^2]}}ds\\
&\leq -\frac{1+\lambda}{2\sqrt{\lambda}}(\varphi_0(\Exp[H_{\mu}(\hx(t))]-\bar{H}_{\mu,t})-\varphi_0(\Exp[H_{\mu}(\hx(T_0))]-\bar{H}_{\mu,T_0})).
\end{aligned}
\ee
By letting $t\to +\infty$, from \eqref{eq:sto-conve-obj} and the continuity of $\varphi_0$, we obtain 
\[
\int_{T_0}^{+\infty}{\Exp[\|\dot{\hx}(s)\|]}ds<+\infty,
\]
which leads to 
\bee
\int_{T_0}^{+\infty}{\|\dot{\hx}(s)\|}ds<+\infty \text{ a.s.}
\eee  
It suggests the existence of  an event $\cA$ with $\Prob(\cA)=1$ such that for any $\omega\in\cA$, 
\bee
\int_{T_0}^{+\infty}{\|\dot{\hx}(s,\omega)\|}ds<+\infty.
\eee  
Then, by Cauchy's criterion, $\hx(t,\omega)$ is convergent for any $\omega\in\cA$. Therefore, $\hx(t)$ converges almost surely.  This, along with Theorem \ref{th:sto-weak-convergence}(iv), indicates that there exists a random vector $x_\infty$ such that $\hx(t)$ converges almost surely to $x_{\infty}$ and $x_\infty\in\crit H_{\mu}$ almost surely.   We thus derive items (i) and (ii). 

By the continuity of $\prox_{\mu h}$, it is clear  that $\bar{\hx}(t)$ converges almost surely to 
\[\bar{x}_{\infty}:=x_{\infty}-A^T(AA^T)^{-1}(Ax_{\infty}-\prox_{\mu h}(Ax_{\infty})).\] 
Since $x_{\infty}\in\crit H_{\mu}$ almost surely,  following Lemma \ref{lem:critical} and the definition of $\epsilon$-approximate critical point, we obtain  $\bar{x}_{\infty}\in\crit_\epsilon H$ almost surely, when  $\mu \le {\epsilon\sqrt{\lambda_{min}(AA^T)}}/{(L_fL_h)}$.  Thus proof of item (iii) is completed.
\end{proof}

\begin{remark} 
Referred to \cite{ABRS2010}, $H_{\mu}$ is a KL function with  \L{}ojasiewicz 
 exponent if both $f$ and $h$ are semialgebraic. 
\end{remark}
Finally, we demonstrate the convergence rates of $\bar{\hx}(t)$ in the context of \L{}ojasiewicz exponent.
\begin{theorem}
Under the conditions of Theorem \ref{th:sto-global}, let $x_{\infty}$ be the limit of $\hx(t)$ in the almost sure sense, 
\[\bar{\hx}(t)=\hx(t)-A^T(AA^T)^{-1}(A\hx(t)-\prox_{\mu h}(A\hx(t)))\]
 and 
 \[
\bar{x}_{\infty}=x_{\infty}-A^T(AA^T)^{-1}(Ax_{\infty}-\prox_{\mu h}(Ax_{\infty})).
 \]
Then the following statements hold:  
	\begin{itemize}
		\item[(i)] if $\theta\in(0,1/2]$, there exist  constants $a_1, b_1>0$ and a time $T_1>0$ such that for $t\geq T_1$,
		\[
		 \Exp[\|\bar{\hx}(t)-\bar{x}_{\infty}\|]\leq a_1\exp{(-b_1(1-\theta) t)};
		\]
		\item[(ii)] if $\theta\in(1/2,1)$, there exist  a  constant $c_1>0$ and  time $T_2>0$ such that for $t\geq T_2$,
		\[ 
		\Exp[\|\bar{\hx}(t)-\bar{x}_{\infty}\|]\leq c_1t^{\frac{1-\theta}{1-2\theta}}.
		\]	
	\end{itemize}	
\end{theorem}
\begin{proof}
From \eqref{eq:diff-obj} and $\lambda \dot{\hx}(t)+\rho^{-1/2}\dot{W}(t)=-\nabla H_{\mu}(\hx(t))$, we have for any $t\geq T_0$,
\be\label{eq:sto-rate1}
\begin{aligned}
\frac{d}{dt}(\Exp[H_{\mu}(\hx(t))]-\bar{H}_{\mu,t})&\leq -\lambda^{-1}\Exp[\|\lambda\dot{\hx}(t)+\rho^{-1/2}\dot{W}(t)\|^2]\\
&=-\lambda^{-1}\Exp[ \|\nabla H_{\mu}(\hx))\|^2]\leq  -\lambda^{-1}\left(\Exp[\|\nabla H_{\mu}(\hx))\|]\right)^2\\ 
&\leq  \frac{-\lambda^{-1}}{\sigma^2(1-\theta)^2}\left(\Exp[H_{\mu}(\hx(t))]-\bar{H}_{\mu,t}\right)^{2\theta},
\end{aligned}
\ee
where the second and  last inequalities are from Cauchy-Schwarz inequality and  \eqref{eq:kl-sto-0}, respectively. 	
By the same line as the proof of Theorem \ref{th:conv-rate}, there exist constants $b_0, b_1, b_2>0$ and $T_1,T_2>0$ such that  
\be\label{eq:sto-conv-1}
\Exp[H_{\mu}(\hx(t))-\bar{H}_{\mu,t}]\leq b_0\exp{(-b_1 t)}, \ t\geq T_1,  \ \text{ for } \theta\in(0,1/2], 
\ee
and 
\be\label{eq:sto-conv-2}
\Exp[H_{\mu}(\hx(t))-\bar{H}_{\mu,t}]\leq b_2 t^{\frac{1}{1-2\theta}}, \ t\geq T_2, \ \text{ for }  \theta\in(1/2,1).
\ee
Further using \eqref{eq:sto-varphi-dot}  and  $\varphi_0(s)=\sigma_0 s^{1-\theta}$,  we obtain  
\bee\label{eq:sto-glo2}
\begin{aligned}
\Exp[\|\hx(t_2)-\hx(t_1)\|]=& \ \Exp\left[\|\int_{t_1}^{t_2}{\dot{\hx}(t)} dt \|\right]\leq\Exp\left[\int_{t_1}^{t_2}{\|\dot{\hx}(t)\|} dt\right]=\int_{t_1}^{t_2}{\Exp[\|\dot{\hx}(t)\|]} dt\\
\leq & \  \frac{\sigma_0(1+\lambda)}{2\sqrt{\lambda}}(\Exp[H_{\mu}(\hx(t_1))]-\bar{H}_{\mu,t_1})^{1-\theta},  
\end{aligned}
\eee 
where $t_2\ge t_1\geq T_0$. 
As $t_2$ tends towards infinity, the subsequent inequality is satisfied:
\[
\Exp[\|\hx(t_1)-\hx_\infty\|]\leq \frac{\sigma_0(1+\lambda)}{2\sqrt{\lambda}}(\Exp[H_{\mu}(\hx(t_1))]-\bar{H}_{\mu,t_1})^{1-\theta}.   
\]
Combining the above inequality with \eqref{eq:sto-conv-1} and \eqref{eq:sto-conv-2}, we have 
\be\label{eq:sto-conv-rate-x1}
\Exp[\|\hx(t)-x_{\infty}\|]\leq \frac{\sigma_0(1+\lambda)b_0^{1-\theta}}{2\sqrt{\lambda}}\exp(-b_1(1-\theta)t) \text{ for } \theta\in(0,1/2]
\ee
and 
\be\label{eq:sto-conv-rate-x2}
\Exp[\|\hx(t)-x_{\infty}\|]\leq \frac{\sigma_0(1+\lambda)b_2^{1-\theta}}{2\sqrt{\lambda}} t^{\frac{1-\theta}{1-2\theta}} \text{ for } \theta\in(1/2,1).
\ee
Recall that we have proven that $\bar{\hx}(t)\to \bar{x}_{\infty}$ in  Theorem \ref{th:sto-global}(iii). Let $\cond(A)$ be the condition number of $A$. Notice that 
$A^T(AA^T)^{-1}$ is a pseudoinverse of $A$, then  by   the definition of condition number, $\|A^T(AA^T)^{-1}\|\|A\|=\cond(A)$. Then,  we derive 
\be\label{eq:relation-x-barx}
\begin{aligned}
\Exp[\|\bar{\hx}(t)-\bar{x}_{\infty}\|]&\leq \Exp[\|\hx(t)-x_{\infty}\|]+\|A^T(AA^T)^{-1}\|\Exp[\|A\hx(t)-Ax_{\infty}\|]\\&\quad\quad+\|A^T(AA^T)^{-1}\|\Exp[\|\prox_{\mu h}(A\hx(t))-\prox_{\mu h}(Ax_{\infty})\|]\\
&\leq\left(1+\frac{2-\mu\varrho}{1-\mu\varrho}\cond(A)\right)\Exp[\|\hx(t)-x_{\infty}\|], 
\end{aligned}
\ee
where the last inequality is due to the $\frac{1}{1-\mu\varrho}$-Lipschitz continuity of $\prox_{\mu h}$\cite[Proposition 3.3]{Axel2021}. Combining \eqref{eq:relation-x-barx} with \eqref{eq:sto-conv-rate-x1} and \eqref{eq:sto-conv-rate-x2}, the convergence rates  of $\|\bar{\hx}(t)-\bar{x}_{\infty}\|$ in expectation can be derived:
\bee
\Exp[\|\bar{\hx}(t)-\bar{x}_{\infty}\|]\leq \left(1+\frac{2-\mu\varrho}{1-\mu\varrho}\cond(A)\right) \frac{\sigma_0(1+\lambda)b_0^{1-\theta}}{2\sqrt{\lambda}}\exp(-b_1(1-\theta)t) \text{ for } \theta\in(0,1/2]
\eee
and 
\bee
\Exp[\|\bar{\hx}(t)-\bar{x}_{\infty}\|]\leq \left(1+\frac{2-\mu\varrho}{1-\mu\varrho}\cond(A)\right)\frac{\sigma_0(1+\lambda)b_2^{1-\theta}}{2\sqrt{\lambda}}t^{\frac{1-\theta}{1-2\theta}} \text{ for } \theta\in(1/2,1).
\eee
We complete this proof.
\end{proof}

\section{Accelerated  LP-SADMM and convergence analysis }\label{sec:acce-ADMM}
In this section, drawing inspiration from the Nesterov's  accelerated gradient method \cite{Nesterov1983AMF}, we propose an accelerated variant of  LP-SADMM   for solving  problem \eqref{eq:finite-p} with  $h$ being $\varrho$-weakly convex, and further derive the continuous-time system   of the proposed algorithm.  We then  explore the global convergence properties of the trajectory.

\subsection{Accelerated LP-SADMM and continuous-time system}
 We present the accelerated LP-SADMM  in Algorithm \ref{alg:accelerated ADMM}, where $\Sto f(\hat x^k)$ is a stochastic approximation to $\nabla f(\hat x^k)$ and $h_\mu$ is the $\mu$-Moreau envelop of $h$.
 \begin{algorithm}
\caption{ Accelerated LP-SADMM}\label{alg:accelerated ADMM}
\begin{algorithmic}[1]
\Require Initial point $(x^0,z^0,u^0) \in \R^n\times\R^m\times\R^m$,  parameters $\eta,\rho>0$, $\tau>\rho\|A^TA\|+1/\eta$ and $0<\mu<1/\varrho$. Let $\hat{x}^0=x^0,\hat{z}^0=z^0,\hat{u}^0=u^0$.
\For{ $k=0,1,2,\ldots$ } 
\State Update $x^{k},z^{k},u^{k}$ as follows:
\begin{subnumcases}{\label{eq:iter-acce}}
			x^{k+1}=\hat{x}^k-\frac{1}{\tau}\left(\Sto f(\hat{x}^k)+\rho A^T(A\hat{x}^k-\hat{z}^k+\frac{\hat{u}^k}{\rho})\right),\label{eq:iter1-acce}\\
			z^{k+1}=\argmin_{z\in\R^m}\left\lbrace h_{\mu}(z)+\langle \hat{u}^k, Ax^{k+1}-z\rangle+\frac{\rho}{2}\|Ax^{k+1}-z\|^2\right\rbrace,\label{eq:iter2-acce}\\
			u^{k+1}=\hat{u}^k+Ax^{k+1}-z^{k+1},\label{eq:iter3-acce}\\
			\hat{u}^{k+1}=u^{k+1}+\alpha_{k+1}(u^{k+1}-u^k),\label{eq:iter31-acce}\\
			\hat{x}^{k+1}=x^{k+1}+\alpha_{k+1}(x^{k+1}-x^k),\label{eq:iter11-acce}\\
			\hat{z}^{k+1}=z^{k+1}+\alpha_{k+1}(z^{k+1}-z^k).\label{eq:iter21-acce}
\end{subnumcases}
\State  Set $k \Leftarrow k+1$.
\EndFor 
\end{algorithmic}
\end{algorithm}


With a little abuse of notation, we continue to  use  $\xi^k:=\Sto f(\hat{x}^k)-\nabla f(\hat{x}^k)$ to denote the gradient noise at the $k$-th iteration of Algorithm \ref{alg:accelerated ADMM}. We assume that the sequence ${\xi^k}$ satisfies Assumption \ref{ass:first-sto}. Under this assumption, it is implied that
\be\label{eq:noise-second-order}
\xi^k={\rho}^{1/4}\left(W(t+\frac{1}{\sqrt{\rho}})-W(t)\right) \text{ for } t=\frac{k}{\sqrt{\rho}}.
\ee
Then under Assumptions \ref{ass:acc-setting} and \ref{ass:first-sto}, we first  establish the continuous dynamical system  of Algorithm \ref{alg:accelerated ADMM}.

Same as previous sections, we will analyze the derivation of the continuous counterpart of Algorithm \ref{alg:accelerated ADMM} with a sufficiently large parameter $\rho$. We set $\alpha_k=\frac{\beta k}{k+\alpha}$ with $\beta=1-\gamma /\sqrt{\rho}>0$ and $\alpha,\gamma>0$, and let $\{x^k\}, \{\hat x^k\}, \{z^k\}, \{\hat z^k\},  \{u^k\}, \{\hat u^k\}$ be generated by Algorithm \ref{alg:accelerated ADMM}. By the optimality condition for  \eqref{eq:iter2-acce}, the following holds
	\bee
	0= \nabla h_{\mu}(z^{k+1})-\rho(Ax^{k+1}-z^{k+1}+\frac{1}{\rho}\hat{u}^k),
	\eee	
	which, together with \eqref{eq:iter1-acce} and $\Sto f(\hat{x}^k)=\nabla f(\hat{x}^k)+\xi^k$, leads to 
	\be\label{eq:optimal-acce}
	0=\tau(x^{k+1}-\hat{x}^k)+\nabla f(\hat{x}^k)+\xi^k+A^{T}\nabla h_{\mu}(z^{k+1})+\rho A^{T}(z^{k+1}-\hat{z}^k)-\rho A^TA(x^{k+1}-\hat{x}^{k}).
	\ee
We introduce the \emph{Ansatz}: $x^k\approx\hx(k/\rho), \hat{\hx}\approx\hat{\hx}(k/\rho),  z^k\approx\hz(k/\rho), \hat{z}^k\approx\hat{\hz}(k/\rho), u^k\approx\hu(k/\rho), \hat{u}^k\approx\hat{\hu}(k/\rho)$ for twice continuously differentiable stochastic processes $\hx$, $\hat{\hx}$, $\hz$, $\hat{\hz}$, $\hu$, $\hat{\hu}$.  Take $k=t/s$ with $s=1/\sqrt{\rho}$. Then, for sufficiently small $s$,  we have $\hx(t)\approx x^{t/s}=x^k, \hx(t+s)\approx x^{(t+s)/s}=x^{k+1}$. 
This conclusion also holds for $\hat{\hx}$, $\hz$, $\hat{\hz}$, $\hu$, $\hat{\hu}$. The setting of $\alpha_k$ can be expressed as 
\be\label{eq:alpha}
	\alpha_k=\frac{\beta k}{k+\alpha}=\frac{ t-t\gamma s }{t+\alpha s} = 1+ O(s).
\ee 
By approximating $x^{k+1}$ and $\hat{x}^{k+1}$ with their first-order Taylor expansions, \eqref{eq:iter11-acce} can be rewritten as:
	\bee
	0=\hat{\hx}(t)+s\dot{\hat{\hx}}(t)+O(s^2)-(\hx(t)+s\dot{\hx}(t)+O(s^2))-\alpha_{k+1}(s\dot{\hx}(t)+O(s^2)),
	\eee
	which, together with $\alpha_k=O(1)$ (from \eqref{eq:alpha}), implies that   
	$\hat{\hx}(t)=\hx(t)+O(s)$. Similarly, by \eqref{eq:iter21-acce} and \eqref{eq:iter31-acce}, $\hat{\hz}(t)=\hz(t)+O(s)$ and $\hat{\hu}(t)=\hu(t)+O(s)$ hold. Using Taylor's theorem for \eqref{eq:iter3-acce} and substituting  $\hat{\hu}(t)=\hu(t)+O(s)$ into \eqref{eq:iter3-acce} yield that 
	\bee
	0=s\dot{\hu}(t)+O(s)-(A\hx(t)+sA\dot{\hx}(t)+O(s^2)-\hz(t)-s\dot{\hz}(t)-O(s^2)),
	\eee
	which leads to 
	\be\label{eq:ax=z-acce}
	A\hx(t)=\hz(t) \text{ as } s\to 0,
	\ee	
 for any $t\in[0,+\infty)$. We further arrive at the following equations: 
 \be\label{eq:adx=dz+os}
 A\dot{\hx}(t)=\dot{\hz}(t) \text{ and } A\ddot{\hx}(t)=\ddot{\hz}(t),
 \ee
 as $ s\to 0$. 
According to  \eqref{eq:noise-second-order}, $\hat{\hx}(t)=\hx(t)+O(s)$ and  $\hat{\hz}(t+s)=\hz(t)+O(s)$,   the relation \eqref{eq:optimal-acce} can be rewritten into 
	\be\label{eq:eq:optimal-acce1}
	\begin{aligned}
	0&=\tau(\hx(t+s)-\hat{\hx}(t))+ \nabla f(\hx(t)+O(s))+A^{T}\nabla h_{\mu}(\hz(t)+O(s))+{s}^{-1/2}\left(W(t+s)-W(t)\right) \\
	&\quad\quad +s^{-2} A^{T}(\hz(t+s)-\hat{\hz}(t))-s^{-2} A^TA(\hx(t+s)-\hat{\hx}(t)).
	\end{aligned}
	\ee
Then, it yields from  Taylor's theorem for $W(t+s)$  that 
	\be\label{eq:optimal-acce2}
	\begin{aligned}
	0&=\tau(\hx(t+s)-\hat{\hx}(t))+ \nabla f(\hx(t)+O(s)) +A^{T}\nabla h_{\mu}(\hz(t)+O(s))+{s}^{1/2}\dot{W}(t)\\
	&\quad\quad+s^{-2} A^{T}(\hz(t+s)-\hat{\hz}(t))-s^{-2} A^TA(\hx(t+s)-\hat{\hx}(t)) +O(s^{3/2}).
    \end{aligned}
	\ee	
	For the  term $\hx(t+s)-\hat{\hx}(t)$ of \eqref{eq:optimal-acce2}, it follows from \eqref{eq:iter11-acce} that 
	\be\label{eq:x-second-taylor}
	\begin{aligned}
		\hx(t+s)-\hat{\hx}(t)&=\hx(t+s)-\hx(t)-\alpha_{k}(\hx(t)-\hx(t-s))\\
		&=s\dot{\hx}(t)+\frac{s^2}{2}\ddot{\hx}(t)+O(s^3)-\alpha_k\left(s\dot{\hx}(t)-\frac{s^2}{2}\ddot{\hx}(t)+O(s^3)\right)\\
		&=s(1-\alpha_k)\dot{\hx}(t)+\frac{s^2(1+\alpha_k)}{2}\ddot{\hx}(t)+(1-\alpha_k)O(s^3),
	\end{aligned}
	\ee
	where the second equality is obtained by applying a second-order Taylor expansion. 
	From the setting of $\alpha_k$, it holds that $s^{-1}(1-\alpha_k)=\gamma+\frac{\alpha}{t}+O(s)$ and $\alpha_k=1+O(s)$, which, together with \eqref{eq:x-second-taylor}, derive 
	\be\label{eq:x-second-taylor1}
	\begin{aligned}
	s^{-2} (\hx(t+s)-\hat{\hx}(t))&=s^{-1}(1-\alpha_k)\dot{\hx}(t)+\frac{1+\alpha_k}{2}\ddot{\hx}(t)+(1-\alpha_k)O(s)\\
	&=(\gamma+\frac{\alpha}{t})\dot{\hx}(t)+O(s)\dot{\hx}(t)+\ddot{\hx}(t)+O(s)\ddot{\hx}(t)+O(s^2)\\
	&=(\gamma+\frac{\alpha}{t})\dot{\hx}(t)+\ddot{\hx}(t)+O(s).
	\end{aligned}
	\ee 
	On the other hand, for $s^{-2}(\hz(t+s)-\hat{\hz}(t))$, it holds that
	\bee
	\begin{aligned}
	s^{-2}(\hz(t+s)-\hat{\hz}(t))=(\gamma+\frac{\alpha}{t})\dot{\hz}(t)+\ddot{\hz}(t)+O(s).
	\end{aligned} 
	\eee
	Inserting  the above two relations into \eqref{eq:optimal-acce2} indicates 
	\be\label{eq:eq:optimal-acce3}	
    \begin{aligned}
	0&=\tau(\hx(t+s)-\hat{\hx}(t))+ \nabla f(\hx(t)+O(s)) +A^{T}\nabla h_{\mu}(\hz(t)+O(s))+{s}^{1/2}\dot{W}(t)\\
    &\quad\quad +(\gamma+\frac{\alpha}{t})A^T\dot{\hz}(t)+A^T\ddot{\hz}(t)-(\gamma+\frac{\alpha}{t})A^T A\dot{\hx}(t)-A^T A\ddot{\hx}(t) +O(s).	
    \end{aligned}
	\ee
	Since $\tau s^2> \|A^T A\| +s^2/\eta$ from the definition of $\tau$ and $1/\rho=s^2$, there exists a constant $\lambda> \|A^TA\|$ such that $\tau s^2=\lambda+O(s^2)$.  That, together with \eqref{eq:x-second-taylor1}, yields 
	\bee
\tau(\hx(t+s)-\hat{\hx}(t))= \tau s^2\cdot \left((\gamma+\frac{\alpha}{t})\dot{\hx}(t)+\ddot{\hx}(t)+O(s)\right)= \lambda\left((\gamma+\frac{\alpha}{t})\dot{\hx}(t)+\ddot{\hx}(t)\right)+O(s).
	\eee 
	Taking the above equality into \eqref{eq:eq:optimal-acce3},  we obtain
	\bee
 \begin{aligned}
	0&=\lambda\ddot{\hx}(t)+\lambda(\gamma+\frac{\alpha}{t})\dot{\hx}(t)+\nabla f(\hx(t)+O(s)) +A^{T}\nabla h_{\mu}(\hz(t)+O(s)) +{s}^{1/2}\dot{W}(t)\\
  &\quad\quad +(\gamma+\frac{\alpha}{t})A^T\dot{\hz}(t)+A^T\ddot{\hz}(t)-(\gamma+\frac{\alpha}{t})A^T A\dot{\hx}(t)-A^T A\ddot{\hx}(t) +O(s).
   \end{aligned}
	\eee
 According to \eqref{eq:ax=z-acce} and \eqref{eq:adx=dz+os}, and following  discussions in  Subsection \ref{subsec:lp-sadmm-system}, for sufficiently small $s>0$, we  consider the following second-order stochastic differential equation 
 \be\label{eq:inclusion-acce}
	0=\lambda\ddot{\hx}(t)+\lambda(\gamma+\frac{\alpha}{t})\dot{\hx}(t)+\nabla H_{\mu}(\hx(t)) +\rho^{-1/4}\dot{W}(t) \text{ with } \hx(0)=x^0,
 \ee
 where $\lambda> \|A^T A\|.$

\begin{remark}\label{re:w=0}
(i) If $\tilde{\nabla} f(\hat x^k) = \nabla f(\hat x^k)$, i.e., $\xi^k=0$ for any $k\ge 0$,  Algorithm \ref{alg:accelerated ADMM} reduces to a  deterministic accelerated LP-ADMM method. The corresponding continuous dynamical system is  given by 
\be\label{eq:acce-ode}
0=\lambda\ddot{\hx}(t)+\lambda(\gamma+\frac{\alpha}{t})\dot{\hx}(t)+\nabla H_{\mu}(\hx(t)),
\ee
which  is in the same form as  the  second-order dynamical system (2) presented in \cite{Radu2020}. 

(ii) Let $\hX_1(t):=\hx(t)$, $\hX_2(t):=\dot{\hx}(t)$, and $\hX(t):=[\hX_1(t);\hX_2(t)]$. Then \eqref{eq:inclusion-acce} can be expressed as the following first-order stochastic differential equation:
\be\label{eq:second-into-first}
0= \lambda\dot{\hX}(t)+[-\lambda \hX_2(t);\lambda(\gamma+\frac{\alpha}{t})\hX_2(t)+\nabla H_{\mu}(\hX_1(t))]+\rho^{-1/4}[0;I]\dot{W}(t).
\ee
\end{remark}

A weak approximation theorem for \eqref{eq:inclusion-acce} can also be provided, similar to Theorem \ref{th:weak approximation}. Similar result is also presented for the stochastic differential equation for momentum SGD in \cite[Theorem 14]{Li2017icml}.
\begin{theorem}\label{th:weak approximation-second}
Under the conditions of Theorem \ref{th:weak approximation}, 
the stochastic process  $\left\{[\hx(t);\dot{\hx}(t)]:t\in[0,T]\right\}$ satisfying \eqref{eq:inclusion-acce}  is an order $1$ weak approximation of Algorithm  \ref{alg:accelerated ADMM}. 
\end{theorem}
Regarding the existence and uniqueness of a solution of \eqref{eq:inclusion-acce}, note that \eqref{eq:inclusion-acce} is equivalent to \eqref{eq:second-into-first} as stated in Remark \ref{re:w=0}(ii), and  by \cite[Theorem 5.2.1]{Oksendal2003}, \eqref{eq:second-into-first} has a unique solution if the term $[-\lambda \hX_2(t);\lambda(\gamma+\frac{\alpha}{t})\hX_2(t)+\nabla H_{\mu}(\hX_1(t))]$ is Lipschitz continuous with respect to $\hX(t)$. This condition is satisfied due to the Lipschitz continuity of $\nabla H_{\mu}$, which is guaranteed by Assumption \ref{ass:acc-setting}, as demonstrated in Remark \ref{re:moreau}(ii).

\subsection{Convergence analysis}\label{sec:conver-analysis-acce}
Lemma \ref{lem:critical} characterizes the relation between the critical points of $H$ and $H_{\mu}$ under Assumption \ref{ass:acc-setting}. In the following, for any given $x\in \crit H_\mu$,  we will continue to use $$\bar{x}:=x-A^T(AA^T)^{-1}(Ax-\prox_{\mu h}(Ax))$$ to represent an approximation to a critical point of $H$ for small $\mu$. Therefore, to establish that $\bar{\hx}$, which is constructed in \eqref{bar-x}, converges almost surely to an $\epsilon$-approximate critical point of $H$, it is necessary to derive the almost sure convergence of $\hx$ to the critical point of $H_\mu$ under an appropriate setting of $\mu$. Here, $\hx$ is a component of the trajectory $[\hx;\dot{\hx}]$ generated by \eqref{eq:inclusion-acce}. In this case,  however,  characterizing the descent property of $H_{\mu}$ becomes challenging. To address this issue, we need to  introduce an auxiliary function. Following a similar analysis to \cite{Radu2020}, we define
\[\cL_{\mu}(x,y):=H_{\mu}(x)+\frac{\lambda}{2}\|x-y\|^2,\] 
where $x,y\in\R^n$, $\mu\in(0,1/\varrho)$ and $\lambda>\|A^T A\|$.  

The lemma below presents a relationship between the critical points of $H_{\mu}$ and $\cL_{\mu}$, which is inherent in the definition of a critical point. Based on this relationship, we will leverage $\mathcal L_\mu$ to identify and investigate critical points of $H_\mu$.

\begin{lemma}\label{lem:critical-ly}
	For any $x, y\in\R^n$,  $(x,y)\in\crit \cL_{\mu}$ is equivalent to $y=x\in\crit H_{\mu}$.  
\end{lemma}

  The following lemma provides a descent property of $\cL_\mu$.

\begin{lemma}\label{lem:descent-acce}
	Suppose that  Assumptions \ref{ass:acc-setting} and \ref{ass:first-sto} hold, and let $[\hx(t);\dot{\hx}(t)]$ be a bounded trajectory of \eqref{eq:inclusion-acce}. Define \[\hp(t):=c\dot{\hx}(t)+\hx(t) \quad \mbox{and}\quad \hq(t):=\left(c+\sqrt{1+c\gamma+\frac{c\alpha}{t}}\right)\dot{\hx}(t)+\hx(t),
	\]
where $0<c<\min\{\frac{2\lambda}{L}, \frac{\sqrt{L^2+2\lambda \gamma L}-L}{L}\}$  with $\lambda, \gamma, \alpha$ shown in \eqref{eq:inclusion-acce} and 
$L:=L_f+\max\{\frac{1}{\mu}, \frac{\varrho}{1-\varrho\mu}\}\cdot\|A\|^2$.   
	Then  for any $t_2\ge t_1\ge0$, 
	\be\label{eq:descent-acce}
	\Exp[\cL_{\mu}(\hp(t_2),\hq(t_2))]+\Exp\left[a\int_{t_1}^{t_2}{\|\dot{\hx}(t)\|^2}dt+b\int_{t_1}^{t_2}{\|\ddot{\hx}(t)\|^2}dt\right]\leq\Exp[ \cL_{\mu}(\hp(t_1),\hq(t_1))],
	\ee  	
	where
	$a:=\lambda\gamma-cL-\frac{c^2L}{2}>0$, $b:=c\lambda-\frac{c^2L}{2}>0$. 
\end{lemma}
\begin{proof}
	It follows from the definitions of $\cL_{\mu}$, $\hp(t)$ and $\hq(t)$ that  
	\be\label{eq:describe of L}
	\cL_{\mu}(\hp(t),\hq(t))=H_{\mu}(c\dot{\hx}(t)+\hx(t))+\frac{\lambda}{2}(1+c\gamma+\frac{c\alpha}{t})\|\dot{\hx}(t)\|^2.
	\ee
  We use the chain rule \eqref{eq:sto-chain rule for H} for the composition of $\cL_{\mu}$ and the trajectory $[\hx(t);\dot{\hx}(t)]$, obtaining
	\be\label{eq:differential of L}
	\begin{aligned}
		& \frac{d}{dt}\cL_{\mu}(\hp(t),\hq(t))\\
		&=\langle\nabla H_{\mu}(c\dot{\hx}(t)+\hx(t)),  c\ddot{\hx}(t)+\dot{\hx}(t)\rangle+\lambda(1+c\gamma+\frac{c\alpha}{t})\langle\dot{\hx}(t),\ddot{\hx}(t) \rangle-\frac{\lambda c\alpha}{2t^2} \|\dot{\hx}(t)\|^2\\
		&=\langle\nabla H_{\mu}(c\dot{\hx}(t)+\hx(t))-\nabla H_{\mu}(\hx(t)),  c\ddot{\hx}(t)+\dot{\hx}(t)\rangle+\langle\nabla H_{\mu}(\hx(t)),  c\ddot{\hx}(t)+\dot{\hx}(t)\rangle\\
		&\quad\quad+\lambda(1+c\gamma+\frac{c\alpha}{t})\langle\dot{\hx}(t),\ddot{\hx}(t) \rangle-\frac{\lambda c\alpha}{2t^2} \|\dot{\hx}(t)\|^2\\
		&\leq cL\|\dot{\hx}(t)\|\|c\ddot{\hx}(t)+\dot{\hx}(t)\|+\langle-\lambda\ddot{\hx}(t)-\lambda(\gamma+\frac{\alpha}{t})\dot{\hx}(t)-\rho^{-1/4}\dot{W}(t),c\ddot{\hx}(t)+\dot{\hx}(t)\rangle\\	&\quad\quad+\lambda(1+c\gamma+\frac{c\alpha}{t})\langle\dot{\hx}(t),\ddot{\hx}(t) \rangle-\frac{\lambda c\alpha}{2t^2} \|\dot{\hx}(t)\|^2\\
		&=cL\|\dot{\hx}(t)\|\|c\ddot{\hx}(t)+\dot{\hx}(t)\|-c\lambda \|\ddot{\hx}(t)\|^2-\lambda(\gamma+\frac{\alpha}{t}+\frac{ c\alpha}{2t^2})\|\dot{\hx}(t)\|^2\\
		&\quad\quad-\rho^{-1/4}\langle\dot{W}(t),c\ddot{\hx}(t)+\dot{\hx}(t)\rangle,
	\end{aligned}
	\ee	
	where the above inequality is deduced from the Lipschitz continuity of $\nabla H_\mu$ (Remark \ref{re:moreau}(ii)) and \eqref{eq:inclusion-acce}.  Applying inequality $2\langle x,y\rangle\leq \|x\|^2+\|y\|^2$ for any  $x,y\in\R^n$, yields that 
	\bee
	\|\dot{\hx}(t)\|\|c\ddot{\hx}(t)+\dot{\hx}(t)\|\leq c\|\dot{\hx}(t)\|\|\ddot{\hx}(t)\|+\|\dot{\hx}(t)\|^2\leq (1+\frac{c}{2})\|\dot{\hx}(t)\|^2+\frac{c}{2}\|\ddot{\hx}(t)\|^2.
	\eee
	Substituting the above inequality into \eqref{eq:differential of L} and rearranging it derive the inequality:
	\bee
	\frac{d}{dt}\cL_{\mu}(\hp(t),\hq(t))\leq -a\|\dot{\hx}(t)\|^2-b\|\ddot{\hx}(t)\|^2-\rho^{-1/4}\langle\dot{W}(t),c\ddot{\hx}(t)+\dot{\hx}(t)\rangle.
	\eee
    By taking expectation on both sides of the above inequality, we  have
	\bee
	\frac{d}{dt}\Exp[\cL_{\mu}(\hp(t),\hq(t))]\leq -a\Exp[\|\dot{\hx}(t)\|^2]-b\Exp[\|\ddot{\hx}(t)\|^2]-\rho^{-1/4}\Exp[\langle\dot{W}(t),c\ddot{\hx}(t)+\dot{\hx}(t)\rangle].	
	\eee
    From the boundedness of a trajectory, it follows that $c\ddot{\hx}+\dot{\hx}\in \mathbb L_e^2(0, t)$ for any  $t> 0$. Thus,  $\int_{0}^{t}{\langle\dot{W}(s),c\ddot{\hx}(s)+\dot{\hx}(s)\rangle}ds$ is an It{\^o}'s integral,   
    and 
    \[\Exp\left[\int_{0}^{t}{\langle\dot{W}(s),c\ddot{\hx}(s)+\dot{\hx}(s)\rangle}ds\right]=0,\] 
    which further implies  that $\Exp[\langle\dot{W}(t),c\ddot{\hx}(t)+\dot{\hx}(t)\rangle]=0$ for almost every $t> 0$. Hence, the following inequality holds almost everywhere:
	\be\label{eq:descent-acce1}
	\frac{d}{dt}\Exp[\cL_{\mu}(\hp(t),\hq(t))]\leq -a\Exp[\|\dot{\hx}(t)\|^2]-b\Exp[\|\ddot{\hx}(t)\|^2].	
	\ee	
	Then, integrating the above inequality over $[t_1,t_2]$  leads to \eqref{eq:descent-acce}. The  definition of $c$ indicates that the constants $a,b$ are positive.  We complete this proof.  
\end{proof}

Using the previously established descent property of $\cL_{\mu}$, we proceed to investigate a preliminary convergence result.
\begin{theorem}\label{th:weak-conver1-acce}
Suppose that Assumptions \ref{ass:acc-setting} and \ref{ass:first-sto} hold, and let $[\hx(t);\dot{\hx}(t)]$ be  a bounded trajectory  of \eqref{eq:inclusion-acce}. Then it holds that
	\[\int_{0}^{\infty}{\|\dot{\hx}(t)\|^2} dt<\infty\text{ a.s.}, \  \int_{0}^{\infty}{\|\ddot{\hx}(t)\|^2} dt<\infty \text{ a.s.}\]
	and $\dot{\hx}(t)\to 0$ almost surely.
\end{theorem}
\begin{proof}
Assumption \ref{ass:acc-setting} imposes a lower bound on $f$ and $h$, making it straightforward to establish the lower boundedness of $\cL_{\mu}$, defined in \eqref{eq:describe of L}. This, along with the nonincreasing property of $\Exp[\cL_{\mu}(\hp(t),\hq(t))]$ demonstrated in Lemma \ref{lem:descent-acce}, suggests the existence of a finite constant $\bar{\cL_{\mu}}$ such that 
	\be\label{eq:obj-conv-acce}
	\Exp[\cL_{\mu}(\hp(t),\hq(t))]\to \bar{\cL}_{\mu} \text{ as } t\to \infty.
	\ee	
  It is noteworthy  from the boundedness  of trajectory that  $\dot{\hx}(t), \ddot{\hx}(t)\in\mathbb L_e^2(t_1, t_2)$  for any  $t_1,t_2\ge 0$, thus we can interchange the order of integration and expectation in \eqref{eq:descent-acce}, obtaining
	\bee
a\int_{t_1}^{t_2}{\Exp\left[\|\dot{\hx}(t)\|^2\right]}dt+b\int_{t_1}^{t_2}{\Exp\left[\|\ddot{\hx}(t)\|^2\right]}dt\leq\Exp[ \cL_{\mu}(\hp(t_1),\hq(t_1))]-	\Exp[\cL_{\mu}(\hp(t_2),\hq(t_2))].
	\eee 	
	Let $t_1=0$ and $t_2\to +\infty$ in the above inequality. Then it, together with  \eqref{eq:obj-conv-acce}, implies 
	\be\label{eq:dotx-l2}
	\int_{0}^{+\infty}{\Exp\left[\|\dot{\hx}(t)\|^2\right]}dt<+\infty \text{ and } \int_{0}^{+\infty}{\Exp\left[\|\ddot{\hx}(t)\|^2\right]}dt<+\infty,
	\ee
	which further indicates 
	\bee\label{eq:ddotx-l1}
	\int_{0}^{+\infty}{\|\dot{\hx}(t)\|^2}dt<\infty \text{ a.s.}\text{ and } \int_{0}^{+\infty}{\|\ddot{\hx}(t)\|^2}dt<\infty\text{ a.s.}
	\eee
   Therefore, by \cite[Lemma 4]{Radu2018}, it follows from   $\frac{d}{dt}(\|\dot{\hx}(t)\|^2)\leq \|\dot{\hx}(t)\|^2+\|\ddot{\hx}(t)\|^2$ that $\|\dot{\hx}(t)\|\to 0$ almost surely. 
\end{proof}

Next, we will derive a weak convergence result.
\begin{theorem}\label{th:weak-conv2-acce}
	Suppose that   Assumptions \ref{ass:acc-setting} and \ref{ass:first-sto} hold. Let $[\hx(t);\dot{\hx}(t)]$ be  a bounded trajectory generated by  \eqref{eq:inclusion-acce} and $\cC$ be the set consisting of all  cluster points of  $\{(\hp(t),\hq(t))\}$. Then 
the following statements hold true:
	\begin{itemize}
		\item[(i)] $\cC$ is nonempty, almost surely compact and 
		\[\dist((\hp(t),\hq(t)),\cC)\to 0;\]
       \item[(ii)] $\cL_{\mu}$ has constant expectation over  $\cC$;
	  \item[(iii)]  $\cC\subseteq\crit \cL_{\mu}$ holds almost surely. 
	\end{itemize}
\end{theorem}
\begin{proof}
	Since $\hx(t)$ and $\dot{\hx}(t)$ are bounded, by the definitions of $\hp(t)$ and $\hq(t)$ as presented in  Lemma \ref{lem:descent-acce}, $\{(\hp(t),\hq(t))\}$ is bounded, which further indicates that $\cC$ is nonempty and almost surely compact following the proof of Theorem \ref{th:weak-convergence}(iv). In addition, it is obvious that $\dist((\hp(t),\hq(t)),\cC)\to 0$. We obtain item (i).
	
	For any $(\bar{p},\bar{q})\in\cC$, we prove that $(\bar{p},\bar{q})\in\crit \cL_{\mu}$ almost surely, i.e., $\nabla \cL_{\mu}(\bar{p},\bar{q})=0$ almost surely.  By the definition of a cluster point, there exists an increasing  subsequence $\{t_k\}$  such that 
	\be\label{eq:conv-p-q}
 \begin{aligned}
	& \hp(t_k)=c\dot{\hx}(t_k)+\hx(t_k)\to \bar{p} \text{ a.s.}, \\
&  \hq(t_k)=\left(c+\sqrt{1+c\gamma+\frac{c\alpha}{t_k}}\right)\dot{\hx}(t_k)+\hx(t_k)\to \bar{q}\text{ a.s.},
\end{aligned}
	\ee 
	which, together with $\dot{\hx}(t)\to 0$ almost surely (refer to Theorem \ref{th:weak-conver1-acce}), yields that  
	\be\label{eq:cluster of x_tk}
	\hx(t_k)\to \bar{p}\text{ a.s.} \text{ and }\  \bar{p}=\bar{q}\text{ a.s.}
	\ee  
	By the continuity of $\nabla H_{\mu}$ and $\nabla\cL_{\mu}$ and using  \eqref{eq:conv-p-q} and \eqref{eq:cluster of x_tk}, we have 
	\bee
 \begin{aligned}
	& \nabla\cL_{\mu}(\hp(t_k),\hq(t_k)) \\
 &=(\nabla H_{\mu}(\hp(t_k))+\lambda(\hp(t_k)-\hq(t_k)), \lambda(\hq(t_k)-\hp(t_k)))\\
 &\to (\nabla H_{\mu}(\bar{p}),0) \text{ a.s.}
 \end{aligned}
	\eee
	and 
	\bee
	\nabla\cL_{\mu}(\hp(t_k),\hq(t_k))\to \nabla \cL_{\mu}(\bar{p},\bar{q}) \text{ a.s.}
	\eee
	Thus, $\nabla \cL_{\mu}(\bar{p},\bar{q})=(\nabla H_{\mu}(\bar{p}),0)$ almost surely. To show $\nabla \cL_{\mu}(\bar{p},\bar{q})=0$ almost surely, it suffices to demonstrate $\nabla H_{\mu}(\bar{p})=0$ almost surely.  By \eqref{eq:inclusion-acce}, we have 
	\be\label{eq:acce-gradient-H}
	\begin{aligned}
	& \|\nabla H_{\mu}(\hx(t))\|^2 
 \\
 &=\langle \nabla H_{\mu}(\hx(t)), -\lambda\ddot{\hx}(t)-\lambda(\gamma+\frac{\alpha}{t})\dot{\hx}(t)-\rho^{-1/4}\dot{W}(t)\rangle\\
	&\leq \langle \nabla H_{\mu}(\hx(t)), -\rho^{-1/4}\dot{W}(t)\rangle+\frac{1}{2}\|\nabla H_{\mu}(\hx(t))\|^2+\frac{1}{2}\|\lambda\ddot{\hx}(t)+\lambda(\gamma+\frac{\alpha}{t})\dot{\hx}(t)\|^2\\
	&\leq \langle \nabla H_{\mu}(\hx(t)), -\rho^{-1/4}\dot{W}(t)\rangle+\frac{1}{2}\|\nabla H_{\mu}(\hx(t))\|^2+\|\lambda\ddot{\hx}(t)\|^2+\|\lambda(\gamma+\frac{\alpha}{t})\dot{\hx}(t)\|^2.
	\end{aligned}
	\ee
	Integrating the above inequality over $[0,t]$ for any $t\ge 0$ implies
\bee
\begin{aligned}
& \Exp\left[	\int_{0}^{t}{	\|\nabla H_{\mu}(\hx(s))\|^2}ds\right] \\ & \leq 2\Exp\left[	\int_{0}^{t}{ \langle \nabla H_{\mu}(\hx(s)), -\rho^{-1/4}\dot{W}(s)\rangle}ds\right]+2\Exp\left[\int_{0}^{t}{ \|\lambda\ddot{\hx}(s)\|^2+\|\lambda(\gamma+\frac{\alpha}{s})\dot{\hx}(s)\|^2}ds\right]\\
&=2\Exp\left[\int_{0}^{t}{ \|\lambda\ddot{\hx}(s)\|^2}ds+\int_{0}^{t}{\|\lambda(\gamma+\frac{\alpha}{s})\dot{\hx}(s)\|^2 }ds\right],
\end{aligned}
\eee
where the above equality is due to the property of the It{\^o}'s integral.  Following the fact that	the order of expectation and integration of the above inequality is interchangeable and applying \eqref{eq:dotx-l2}, we obtain 
  \[	\int_{0}^{+\infty}{	\Exp\left[\|\nabla H_{\mu}(\hx(s))\|^2\right]}ds <+\infty,\]
which further implies that 
  \[\int_{0}^{+\infty}{	\|\nabla H_{\mu}(\hx(s))\|^2}ds<+\infty \text{ a.s. }\]
From the proof of \cite[corollary 3]{Duchi2018},  the above inequality suggests that 
	\be\label{eq:nablaH to0}
	\nabla H_{\mu}(\hx(t_k))\to 0 \text{ a.s.}
	\ee
According to \eqref{eq:cluster of x_tk} and the continuity of $\nabla H_{\mu}$, it holds that $\nabla H_{\mu}(\hx(t_k))\to\nabla H_{\mu}(\bar{p})$ almost surely, which, together with \eqref{eq:nablaH to0}, leads to $\nabla H_{\mu}(\bar{p})=0$ almost surely. 
Thus, $\nabla \cL_{\mu}(\bar{p},\bar{q})=0$ almost surely.  Item (iii) is derived. 

Due to the continuity of $\cL_{\mu}$ and \eqref{eq:conv-p-q}, \eqref{eq:cluster of x_tk}, the following relation holds: 
\bee
	\cL_{\mu}(\hp(t_k),\hq(t_k))\to \cL_{\mu}(\bar{p},\bar{q})=H_\mu(\bar{p}) \text{ a.s.}, 
	\eee
  which from Lebesgue's dominated convergence theorem further indicates,
  \bee
  \Exp[\cL_{\mu}(\hp(t_k),\hq(t_k))]\to  \Exp[\cL_{\mu}(\bar{p},\bar{q})]=\Exp[H_\mu(\bar{p})].
  \eee
  Then it together with \eqref{eq:obj-conv-acce} implies $\Exp[H_\mu(\bar{p})]=\bar{\cL}_{\mu}$ for any $(\bar{p},\bar{q})\in\cC$.  Hence, item (ii) holds. 
\end{proof}
A global  convergence and convergence rates can be derived in the rest of this section, under the assumption that $\mathcal L_\mu$ is a KL function.
\begin{theorem}\label{th:global-conv-acce}
	Under Assumptions  \ref{ass:acc-setting} and \ref{ass:first-sto}, suppose that $\cL_{\mu}$ is a KL function with \L{}ojasiewicz exponent $\theta$  and $[\hx(t);\dot{\hx}(t)]$ is  a bounded trajectory generated by \eqref{eq:inclusion-acce}. 
	Then, the following statements hold true:  
	\begin{itemize}
	\item[(i)] there exists a time $T_0> 0$ such that $\int_{T_0}^{\infty}{\|\dot{\hx}(t)\|}dt<+\infty$ a.s.  and 
	$\int_{T_0}^{\infty}{\|\ddot{\hx}(t)\|}dt<+\infty$ a.s.;
	\item[(ii)] $\hx(t)$ converges almost surely to a critical point  of $H_{\mu}$;
	\item[(iii)]   $\bar{\hx}(t)=\hx(t)-A^T(AA^T)^{-1}(A\hx(t)-\prox_{\mu h}(A\hx(t)))$   converges almost surely to  an $\epsilon$-approximate critical point of $H$, if $\mu \le {\epsilon\sqrt{\lambda_{min}(AA^T)}}/{(L_fL_h)}$.
\end{itemize}	
\end{theorem}
\begin{proof}	
Since $\cL_{\mu}$ is a KL function with \L{}ojasiewicz exponent $\theta$, by \cite[Lemma 4.5]{DTLDS2021}, there exist $T_0>0$ and a continuous concave function $\varphi_0(s)=\sigma_0 s^{1-\theta}$ with $\theta\in(0,1)$ such that 
\be\label{eq:kl-acce-sto-0}
\varphi_0'(\Exp[\cL_{\mu}(\hp(t),\hq(t))]-\bar{\cL}_{\mu,t})\cdot\Exp[\|\nabla \cL_{\mu}(\hp(t),\hq(t))\|]\geq 1, \ \forall t \geq T_0.
\ee
Here, $\bar{\cL}_{\mu,t}$ is a nondecreasing sequence and converges to $\bar{\cL}_{\mu}$ which is introduced in \eqref{eq:obj-conv-acce}. This 
implies  
\bee
\begin{aligned}
 \frac{d}{dt}(\Exp[\cL_{\mu}(\hp(t),\hq(t))]-\bar{\cL}_{\mu,t}) 
& = \frac{d}{dt}(\Exp[\cL_{\mu}(\hp(t),\hq(t))]-\bar{\cL}_{\mu})+\frac{d}{dt}(\bar{\cL}_{\mu}-\bar{\cL}_{\mu,t})\\
& \leq \frac{d}{dt}(\Exp[\cL_{\mu}(\hp(t),\hq(t))]-\bar{\cL}_{\mu}).
\end{aligned}
\eee
Substituting \eqref{eq:descent-acce1} into the above inequality, it holds that 
\be\label{eq:conv-acce-1}
\frac{d}{dt}(\Exp[\cL_{\mu}(\hp(t),\hq(t))]-\bar{\cL}_{\mu,t})\leq -a\Exp[\|\dot{\hx}(t)\|^2]-b\Exp[\|\ddot{\hx}(t)\|^2],
\ee
which, together with \eqref{eq:kl-acce-sto-0}, yields   
\be\label{eq:conv-acce-conbine}
\begin{aligned}
&  \frac{d}{dt}\varphi_0(\Exp[\cL_{\mu}(\hp(t),\hq(t))]-\bar{\cL}_{\mu,t}) \\
&=\varphi_0'(\Exp[\cL_{\mu}(\hp(t),\hq(t))]-\bar{\cL}_{\mu,t})\cdot \frac{d}{dt}(\Exp[\cL_{\mu}(\hp(t),\hq(t))]-\bar{\cL}_{\mu,t}) \\
&\leq \frac{-a\Exp[\|\dot{\hx}(t)\|^2]-b\Exp[\|\ddot{\hx}(t)\|^2]}{\Exp[\|\nabla \cL_{\mu}(\hp(t),\hq(t))\|]}.
\end{aligned}
\ee
Note that by the definition of $\cL_{\mu}$, it follows that 
\bee
\begin{aligned}
& \|\nabla \cL_{\mu}(\hp(t),\hq(t))\|^2 \\
&= \|\nabla H_{\mu}(\hp(t))+\lambda(\hp(t)-\hq(t))\|^2+\|\lambda(\hp(t)-\hq(t))\|^2\\
&\leq 2\|\nabla H_{\mu}(c\dot{\hx}(t)+\hx(t))\|^2+3\lambda^2(1+c\gamma+\frac{c\alpha}{t})\|\dot{\hx}(t)\|^2\\
&=2\|\nabla H_{\mu}(c\dot{\hx}(t)+\hx(t))-\nabla H_{\mu}(\hx(t))+\nabla H_{\mu}(\hx(t))\|^2+3\lambda^2(1+c\gamma+\frac{c\alpha}{t})\|\dot{\hx}(t)\|^2\\
&\leq 4\|\nabla H_{\mu}(c\dot{\hx}(t)+\hx(t))-\nabla H_{\mu}(\hx(t))\|^2+4\|\nabla H_{\mu}(\hx(t))\|^2+3\lambda^2(1+c\gamma+\frac{c\alpha}{t})\|\dot{\hx}(t)\|^2\\
&\leq \left(4c^2L^2+3\lambda^2(1+c\gamma+\frac{c\alpha}{t})\right)\|\dot{\hx}(t)\|^2+4\|\nabla H_{\mu}(\hx(t))\|^2,
\end{aligned}
\eee
where the first inequality is obtained from the definitions of $\hp(t)$ and $\hq(t)$ and inequality $\|x+y\|^2\leq 2\|x\|^2+2\|y\|^2$, the last inequality is from the $L$-Lipschitz continuity of $\nabla H_{\mu}$. 
By taking expectation on both sides of the above inequality, we have 
\be\label{eq:conv-acce2}
\Exp[\|\nabla \cL_{\mu}(\hp(t),\hq(t))\|^2]\leq \left(4c^2L^2+3\lambda^2(1+c\gamma+\frac{c\alpha}{t})\right)\Exp[\|\dot{\hx}(t)\|^2]+4\Exp[\|\nabla H_{\mu}(\hx(t))\|^2].
\ee 
Recalling the property of the It{\^o}'s integral, we have, for any $t>0$, 
\[\Exp\left[	\int_{0}^{t}{ \langle \nabla H_{\mu}(\hx(s)), -\rho^{-1/4}\dot{W}(s)\rangle}ds\right]=0,\] 
thus for almost every $t\in [0,+\infty)$, 
\[\Exp\left[ \langle \nabla H_{\mu}(\hx(t)), -\rho^{-1/4}\dot{W}(t)\rangle \right]=0, \] 
which, together with \eqref{eq:acce-gradient-H}, leads to
\be\label{eq:exp-gradient}
\Exp[\|\nabla H_{\mu}(\hx(t))\|^2]\leq 2\lambda^2\Exp[\|\ddot{\hx}(t)\|^2]+2\lambda^2(\gamma+\frac{\alpha}{t})^2\Exp[\|\dot{\hx}(t)\|^2], \text{ for a.e.  } t\in [0,+\infty). 
\ee
Substituting  the above inequality into \eqref{eq:conv-acce2}, we obtain
\bee
\Exp[\|\nabla \cL_{\mu}(\hp(t),\hq(t))\|^2]\leq e(t)\Exp[\|\dot{\hx}(t)\|^2]+8\lambda^2\Exp[\|\ddot{\hx}(t)\|^2], 
\eee
where $e(t):=4c^2L^2+3\lambda^2(1+c\gamma+\frac{c\alpha}{t})+8\lambda^2(\gamma+\frac{\alpha}{t})^2$. Then using Cauchy-Schwarz inequality  we attain
\be\label{eq:acce-gradient-L}
\Exp[\|\nabla \cL_{\mu}(\hp(t),\hq(t))\|]\leq \sqrt{e(t)}\sqrt{\Exp[\|\dot{\hx}(t)\|^2]}+\sqrt{8}\lambda\sqrt{\Exp[\|\ddot{\hx}(t)\|^2]}, 
\ee
which further indicates from  \eqref{eq:conv-acce-conbine} that
\be\label{eq:conv-acce-conbine1}
	\frac{d}{dt}\varphi_0(\Exp[\cL_{\mu}(\hp(t),\hq(t))]-\bar{\cL}_{\mu,t})\leq \frac{-a\Exp[\|\dot{\hx}(t)\|^2]-b\Exp[\|\ddot{\hx}(t)\|^2]}{\sqrt{e(t)}\sqrt{\Exp[\|\dot{\hx}(t)\|^2]}+\sqrt{8}\lambda\sqrt{\Exp[\|\ddot{\hx}(t)\|^2]}}.
\ee
Next, we show that there exists a  constant $\gamma_1>0$ such that 
\be\label{eq:conv-acce-precond}
\frac{-a\Exp[\|\dot{\hx}(t)\|^2]-b\Exp[\|\ddot{\hx}(t)\|^2]}{\sqrt{e(t)}\sqrt{\Exp[\|\dot{\hx}(t)\|^2]}+\sqrt{8}\lambda\sqrt{\Exp[\|\ddot{\hx}(t)\|^2]}}\leq -\gamma_1\left(\sqrt{\Exp[\|\dot{\hx}(t)\|^2]}+\sqrt{\Exp[\|\ddot{\hx}(t)\|^2]}\right).
\ee
As a matter of fact, it is easy to check:
\bee
\begin{aligned}
&\left(\sqrt{e(t)}\sqrt{\Exp[\|\dot{\hx}(t)\|^2]}+\sqrt{8}\lambda\sqrt{\Exp[\|\ddot{\hx}(t)\|^2]}\right)\cdot\left(\sqrt{\Exp[\|\dot{\hx}(t)\|^2]}+\sqrt{\Exp[\|\ddot{\hx}(t)\|^2]}\right)\\
&\leq \frac{3\sqrt{e(t)}+\sqrt{8}\lambda}{2}\Exp[\|\dot{\hx}(t)\|^2]+\frac{\sqrt{e(t)}+3\sqrt{8}\lambda}{2}\Exp[\|\ddot{\hx}(t)\|^2].
\end{aligned}
\eee
We can set 
\[\gamma_1=-\max\left\{\max_{t\geq T_0}\frac{-2a}{3\sqrt{e(t)}+\sqrt{8}\lambda},\max_{t\geq T_0}\frac{-2b}{\sqrt{e(t)}+3\sqrt{8}\lambda}\right\}\]
to ensure \eqref{eq:conv-acce-precond}. Thus, we derive the inequality
\be\label{eq:conv-acce-global1}
\frac{d}{dt}\varphi_0(\Exp[\cL_{\mu}(\hp(t),\hq(t))]-\bar{\cL}_{\mu,t})\leq -\gamma_1\left(\sqrt{\Exp[\|\dot{\hx}(t)\|^2]}+\sqrt{\Exp[\|\ddot{\hx}(t)\|^2]}\right).
\ee
Integrating \eqref{eq:conv-acce-global1} over $[T_0, t]$ with $t>T_0$  yields 
\[
\begin{aligned}
& \gamma_1\int_{T_0}^{t}{\sqrt{\Exp[\|\dot{\hx}(s)\|^2]}+\sqrt{\Exp[\|\ddot{\hx}(s)\|^2] }}ds \\
& \leq  \varphi_0(\Exp[\cL_{\mu}(p(T_0),q(T_0))]-\bar{\cL}_{\mu,T_0})-\varphi_0(\Exp[\cL_{\mu}(\hp(t),\hq(t))]-\bar{\cL}_{\mu,t}).
\end{aligned}
\]
By letting $t\to +\infty$ in the above inequality and by using  the  Cauchy-Schwarz inequality, we obtain 
\bee
\int_{T_0}^{\infty}{\Exp[\|\dot{\hx}(s)\|]}ds<+\infty \text{ and } \int_{T_0}^{\infty}{\Exp[\|\ddot{\hx}(s)\|]}ds
<+\infty,
\eee
which implies  
\bee
\int_{T_0}^{\infty}{\|\dot{\hx}(t)\|}dt<+\infty \text{ a.s.  and } 
	\int_{T_0}^{\infty}{\|\ddot{\hx}(t)\|}dt<+\infty \text{ a.s. }
\eee
Thus item (i) holds ture. Further, there exists an event  $\cA$ with $\Prob(\cA)=1$ such that for any $\omega\in\cA$, 
\bee
\int_{T_0}^{\infty}{\|\dot{\hx}(t,\omega)\|}dt<+\infty \text{ and } 
	\int_{T_0}^{\infty}{\|\ddot{\hx}(t,\omega)\|}dt<+\infty.
\eee
By Cauchy's criterion, we derive $\hx(t,\omega)\to x_{\infty}(\omega)$ for any $\omega\in\cA$. Thus $\hx(t)\to x_{\infty}$ almost surely. Since $\dot{\hx}(t)\to 0$ almost surely  (from Theorem \ref{th:weak-conver1-acce}),   by the definitions of $\hp(t)$ and $\hq(t)$, $(\hp(t),\hq(t))\to (x_{\infty},x_{\infty})$ almost surely, and $(x_{\infty},x_{\infty})\in\crit\cL_{\mu}$ almost surely from Theorem \ref{th:weak-conv2-acce}(iii). Following Lemma \ref{lem:critical-ly}, we further have $x_{\infty}\in\crit H_{\mu}$ almost surely.  Then it follows from the continuity of $\prox_{\mu h}$ that 
\be\label{eq:conv-barhx}  
\bar{\hx}(t)\to \bar{x}_\infty:=x_{\infty}-A^T(AA^T)^{-1}(Ax_{\infty}-\prox_{\mu h}(Ax_{\infty})) \text{ a.s.},
\ee
and from Lemma \ref{lem:critical} and the definition of $\epsilon$-approximate critical point, $\bar{x}_\infty\in\crit_\epsilon H$ almost surely  when $\mu \le {\epsilon\sqrt{\lambda_{min}(AA^T)}}/{(L_fL_h)}$.  Hence, items (ii) and (iii) hold true. 
\end{proof}

\begin{theorem}\label{th:conv-rate-acce}
	Under the conditions of Theorem \ref{th:global-conv-acce}, suppose that  $\cL_{\mu}$ is a KL function with \L{}ojasiewicz exponent $\theta$. Let  $x_{\infty}$ be the limit of $\hx(t)$ in the almost sure sense, 
	\[\bar{\hx}(t)=\hx(t)-A^T(AA^T)^{-1}(A\hx(t)-\prox_{\mu h}(A\hx(t)))\]
 and 
 \[
\bar{x}_{\infty}:=x_{\infty}-A^T(AA^T)^{-1}(Ax_{\infty}-\prox_{\mu h}(Ax_{\infty})).
 \]
	Then the following statements hold true:  
	\begin{itemize}
		\item[(i)] if $\theta\in(0,1/2]$, there exist  constants $a_1, b_1>0$ and time $T_1> 0$ such that for $t\geq T_1$,
		\[
		  \Exp[\|\bar{\hx}(t)-\bar{x}_{\infty}\|]\leq a_1\exp(-b_1(1-\theta)t);
		\]
		\item[(ii)] if $\theta\in(1/2,1)$, there exists a  constant $c_1>0$ and time $T_2> 0$ such that for $t\ge T_2$,
		\[
		\Exp[\|\bar{\hx}(t)-\bar{x}_{\infty}\|]\leq c_1t^{\frac{1-\theta}{1-2\theta}}.
		\]			
	\end{itemize}
\end{theorem}
\begin{proof}
Since $\varphi_0(s)=\sigma_0 s^{1-\theta}$,  we obtain from \eqref{eq:conv-acce-global1} that 
\be\label{eq:acce-rate-1}
\frac{d}{dt}(\Exp[\cL_{\mu}(\hp(t),\hq(t))]-\bar{\cL}_{\mu,t})^{1-\theta}\leq -\frac{\gamma_1}{\sigma_0}\left(\sqrt{\Exp[\|\dot{\hx}(t)\|^2]}+\sqrt{\Exp[\|\ddot{\hx}(t)\|^2]}\right).
\ee	
From \eqref{eq:kl-acce-sto-0} and \eqref{eq:acce-gradient-L}, it follows that
\be\label{eq:acce-rate-2}
\begin{aligned}
\frac{1}{\sigma_0(1-\theta)}(\Exp[\cL_{\mu}(\hp(t),\hq(t))]-\bar{\cL}_{\mu,t})^\theta &\leq  \Exp[\|\nabla \cL_{\mu}(\hp(t),\hq(t))\|] \\
&\leq  \gamma_2\left(\sqrt{\Exp[\|\dot{\hx}(t)\|^2]}+\sqrt{\Exp[\|\ddot{\hx}(t)\|^2]}\right),
\end{aligned}
\ee
where $\gamma_2:=\max\{\max_{t\geq T_0}\sqrt{e(t)},\sqrt{8}\lambda\}$. Combining \eqref{eq:acce-rate-1} and \eqref{eq:acce-rate-2}, we have  
\bee
\frac{d}{dt}(\Exp[\cL_{\mu}(\hp(t),\hq(t))]-\bar{\cL}_{\mu,t})^{1-\theta}\leq\frac{-\gamma_1}{\sigma_0^2\gamma_2(1-\theta)}(\Exp[\cL_{\mu}(\hp(t),\hq(t))]-\bar{\cL}_{\mu,t})^\theta.
\eee
Rearranging the above inequality yields
\bee
\frac{d}{dt}(\Exp[\cL_{\mu}(\hp(t),\hq(t))]-\bar{\cL}_{\mu,t})\leq \frac{-\gamma_1}{\sigma_0^2\gamma_2(1-\theta)^2}(\Exp[\cL_{\mu}(\hp(t),\hq(t))]-\bar{\cL}_{\mu,t})^{2\theta},
\eee
which is similar to \eqref{eq:sto-rate1}. 
Then by applying a similar approach, we derive the convergence rates of $\bar{\hx}$ in expectation.
\end{proof}

\section{Conclusion and discussion}
In this paper, we present an analysis of the convergence properties of linearized proximal ADMM methods for fully nonconvex composite optimization, from a dynamic perspective. Firstly, we establish a fundamental connection between the LP-ADMM algorithm and a first-order  differential inclusion. This connection enables us to demonstrate that  the subsequence of the generated trajectory converges to a critical point of  objective function. We then show the global convergence and convergence rates within the context of KL property. To further analyze the convergence of a stochastic variant of LP-ADMM (LP-SADMM), we employ a similar approach to obtain a first-order stochastic differential equation by utilizing  Brownian motion. This allows us to investigate the convergence properties of the algorithm from the viewpoint of  stochastic dynamical system. Additionally, we propose an accelerated LP-SADMM, which combines  Nesterov's  accelerated gradient method with LP-SADMM. We derive the continuous counterpart of this algorithm, resulting in a second-order  differential equation. By leveraging the KL property, the almost sure convergence of the generated trajectory is established. We utilize the almost sure convergence of this trajectory to construct a stochastic process that converges almost surely to an approximate critical point of objective function.  We also derive the  expected convergence rates associated with this stochastic process.

Our analysis of LP-SADMM in Section \ref{sec:sto-ADMM} reveals  that the convergence results remain valid even when the function $h$ is a general nonconvex function, eliminating the need of the Moreau envelope. However, in the case of the accelerated LP-SADMM studied in Section \ref{sec:acce-ADMM}, its convergence depends  upon the smoothness of the objective function. Consequently, we introduced the utilization of the Moreau envelope to derive a smooth approximation $h_{\mu}$ for weakly convex functions $h$. Unfortunately, directly extending the current analysis presented in Section \ref{sec:acce-ADMM} to encompass general nonconvex functions appears to be unattainable currently. It deserves further study in the future work. 


\bmhead{Acknowledgments}

This work was partially supported by the National Key R\&D Program of China (No. 2022YFA1004000), the Major Key Project of PCL (No. PCL2022A05)  and the National Natural Science Foundation of China (Nos. 11271278 and 12271076).

\section*{Declarations}
The authors have not disclosed any competing interests.



\bibliography{sn-bibliography}

\end{document}